\documentclass[11pt]{article}
\usepackage{amsfonts}
\usepackage{amsmath}
\usepackage{latexsym,amsfonts,amssymb,amsmath,longtable,amsthm}
\usepackage[a4paper,left=32mm,top=32mm,right=32mm, bottom=32mm]{geometry}
\usepackage{graphicx}
\usepackage[cp866]{inputenc}
\usepackage[english]{babel}
\usepackage{makeidx}
\usepackage{latexsym,amsfonts,amssymb,amsmath,longtable,amsthm}
\input amssym.def
\usepackage{latexsym}
\usepackage{graphicx}
\usepackage[all]{xy}
\usepackage{amsfonts}
\usepackage{amsmath}
\usepackage{mathrsfs}
\usepackage{color}
\usepackage{ulem}
\usepackage[table]{xcolor}
\usepackage[affil-it]{authblk}

\batchmode

\def\bfeta{\boldsymbol{\eta}}
\def\bxi{\boldsymbol{\xi}}

\def\int{\intop}

\def\sR{{\mathbb{R}}}

\def\div{\hbox{\rm div}\,}

\newtheorem{theo}{\bf Theorem}
\newtheorem{lem}{\bf Lemma}
\newtheorem{remark}{Remark}
\newtheorem{defi}{\bf Definition}

\newcommand{\meas}{\mathop{\rm meas}}

\newcommand{\supp}{\mathop{\rm supp}}

\renewcommand{\div}{\mathop{\rm div}}

\newcommand{\F}{\mathbb{F}}

\newcommand{\nn}{\mathbf{n}}

\title{On Nonhomogeneous Boundary Value Problems for the Stationary  Navier--Stokes Equations in 2D Symmetric Semi-Infinite Outlets}

\author[b]{M. Chipot}

\author[a,b]{K. Kaulakyt\.{e}}

\author[a]{K. Pileckas}
  
\author[b]{W. Xue}

\affil[a]{Faculty of Mathematics and Informatics, Vilnius University\\
Naugarduko str. 24, LT-03225 Vilnius, Lithuania}

\affil[b]{Institute of Mathematics, University of Zurich\\
Winterthurerstrasse 190, CH-8057 Zurich, Switzerland}  
       
\date{ }

\batchmode

\begin{document}
\maketitle
\begin{abstract}
We study the stationary nonhomogeneous Navier--Stokes problem in a two dimensional symmetric domain with a semi-infinite outlet (for instance, either parabo-\\loidal or channel-like). Under the symmetry assumptions on the domain, boundary value and external force we prove the existence of at least one weak symmetric solution without any restriction on the size of the fluxes, i.e. the fluxes of the boundary value ${\bf a}$ over the inner and the outer boundaries may be arbitrarily large. Only the necessary compatibility condition (the total flux is equal to zero) has to be satisfied. Moreover, the Dirichlet integral of the solution can be finite or infinite depending on the geometry of the domain.
\end{abstract}

\setcounter{equation}{0}
\section{Introduction}

In this paper we study the steady Navier--Stokes equations with nonhomogeneous boundary conditions
\begin{equation}\label{prad0}\left\{\begin{array}{rcl}
-\nu \Delta{\bf u}+\big({\bf u}\cdot \nabla\big){\bf u} +\nabla p &
= & {\bf f}\qquad \hbox{\rm in }\;\;\Omega,\\
\div\,{\bf u} &  = & 0  \qquad \hbox{\rm in }\;\;\Omega,\\
 {\bf u} &  = & {\bf a} \qquad \hbox{\rm on }\;\;\partial\Omega
 \end{array}\right\}
 \end{equation}
in a two dimensional symmetric\footnote{For the definition of a symmetric domain see \eqref{sd}.} multiply connected domain $\Omega,$ having one outlet to infinity (paraboloidal or channel-like), where the vector-valued function
${\bf u}={\bf u}(x)$ is the unknown velocity field, the scalar
function $p=p(x)$ is the pressure of
the fluid, while the vector-valued functions ${\bf a}={\bf a}(x)$ and ${\bf f}={\bf f}(x)$ denote the given boundary value and the external force; $\nu>0$ is the viscosity constant of the given fluid. The boundary $\partial\Omega$ consists of an infinite connected outer boundary and finitely many connected components, forming the inner boundary. The fluxes of the boundary value ${\bf a}$ over each component of the inner boundaries and over the outer boundary may be arbitrarily large.

Let us consider firstly the steady Navier--Stokes problem \eqref{prad0} in a bounded domain $\Omega$
with multiply connected Lipschitz boundary $\partial\Omega$ consisting of $N$ disjoint components $\Gamma_j, j=1,...,N.$ The continuity equation $(\ref{prad0}_2)$ implies
the necessary compatibility condition for the solvability of the problem
\eqref{prad0}:
\begin{equation}\label{nc}\begin{array}{l}
\int\limits_{\partial\Omega}{\bf a}\cdot{\bf
n}\,dS=\sum\limits_{j=1}^N\int\limits_{\Gamma_j}{\bf a}\cdot{\bf
n}\,dS=0,
\end{array}
\end{equation}
where ${\bf n}$ is the unit vector of the outward normal to $\partial\Omega$. This condition means that the total flux is zero. 
Starting from the famous paper of J. Leray published in 1933 (see \cite{Leray}) the problem \eqref{prad0} has been extensively studied. Nevertheless, for a long time the
existence of a weak solution ${\bf u}\in W^{1,2}(\Omega)$  to
problem \eqref{prad0} was only proved either under the  condition of zero fluxes
\begin{equation}\label{3}\begin{array}{l}
{\F}_j=\int\limits_{\Gamma_j}{\bf a}\cdot{\bf n}\,dS=0,\qquad
j=1,2,\ldots,N,
\end{array}
\end{equation}
(e.g., \cite{Leray}, \cite{Lad}, \cite{Lad1},
\cite{VorJud}), or assuming the fluxes $\F_j$ to be sufficiently
small (e.g., \cite{BOPI}, \cite{Finn}, \cite{Fu}, \cite{Galdi1},
\cite{Kozono}),  or under certain symmetry assumptions on the domain
$\Omega$ and the boundary value ${\bf a}$ (e.g., {\cite{Amick},
\cite{Fu1}, \cite{FM}, \cite{Morimoto}, \cite{Pukhnachev},
\cite{Pukhnachev1}, \cite{Sazonov}, \cite{KPR1}). We call \eqref{nc} the general outflow condition and \eqref{3} - the stringent outflow condition.
However, the fundamental question (formulated by J. Leray in
\cite{Leray}) whether problem \eqref{prad0} is solvable
only under the necessary compatibility condition \eqref{nc} (this so called {\it Leray's problem}) had been open for 80 years. However a huge progress has been made recently. The Leray problem was solved for 2-dimensional bounded multiply connected domain (see \cite{KPR}, \cite{KPR0}, \cite{KPR5}).

Nevertheless, not much is known about the
nonhomogeneous boundary value problem \eqref{prad0} in unbounded domains. The first time in 1999 S.A. Nazarov and K. Pileckas  solved problem \eqref{prad0} in an infinite layer without the smallness assumption on the flux of the boundary value, i.e. on the bottom of the layer there is a compactly supported sink or source of an arbitrary intensity (see \cite{NazPil3}). Later in 2010 J. Neustupa \cite{Neustupa1},
\cite{Neustupa} studied problem \eqref{prad0} in unbounded
domains $\Omega$ with multiply connected boundaries under the ``smallness" assumption of the fluxes of ${\bf a}$ over bounded components of the boundary (he did not impose any conditions on fluxes over
infinite parts of $\partial\Omega.$). However,
the solutions found in \cite{Neustupa1},
\cite{Neustupa} have finite Dirichlet integrals (notice that the a priori estimate of solutions was obtained by a contradiction argument).
Recently, problem \eqref{prad0} has been studied in a class of
domains $\Omega\subset {\mathbb{R}}^{n}, \ n=2,3,$ having paraboloidal and layer type 
outlets to infinity (see \cite{KAU}, \cite{KP}). In \cite{KAU}, \cite{KP} it is assumed that the
fluxes of ${\bf a}$ over the bounded connected components of the inner boundary
are sufficiently small while there are no
restrictions on the fluxes of boundary value ${\bf a}$ over noncompact
connected components of the outer boundary. Under these conditions the existence of at least one weak solution to problem
\eqref{prad0} was proved. This solution can have
either finite or infinite Dirichlet integral depending on geometrical
properties of the outlets. The proofs in \cite{KAU}, \cite{KP} are based on a
special construction of the extension of the boundary
value ${\bf a}$ which satisfies Leray--Hopf's inequality and allows
to get effective estimate of the solution.

H. Fujita and H. Morimoto
(see \cite{Morimoto1}--\cite{Morimoto4}) have solved problem \eqref{prad0} in a symmetric two dimensional multiply connected domains $\Omega$ with
channel-like  outlets to infinity containing a finite number of
``holes". Under certain symmetry assumptions on domain, boundary value and external force, in \cite{Morimoto1}--\cite{Morimoto4} the authors also assumed that the boundary value ${\bf a}$ is equal to zero on the
outer boundary and that in each outlet the flow tends to a Poiseuille flow. Moreover, the fluxes over the boundary of each ``hole" may be arbitrarily large, but the sum of them has to be equal to the flux of the corresponding Poiseuille flow which needs to be sufficiently small. In addition the viscosity of the fluid has to be relatively large.

In this paper we prove the existence of at least one weak symmetric solution to problem \eqref{prad0} in a symmetric domain $\Omega\subset {\mathbb{R}}^{2}$ with either a paraboloidal or a channel-like outlet to infinity assuming that the boundary value ${\bf a}$ and the external force $\bf{f}$ are symmetric functions.
Notice that we do not impose any restrictions on the size of the fluxes over both the inner and the outer boundaries.

\setcounter{lem}{0}
\setcounter{equation}{0}
\section{Main Notation and Auxiliary Results}

Vector valued functions are denoted by bold letters while function spaces for scalar and vector valued functions are denoted the same way.  

Let $\Omega$ be a domain in ${\mathbb{R}}^n$.
$C^{\infty}(\Omega)$ denotes the set of all infinitely
differentiable functions defined  on $\Omega$ and
$C_0^{\infty}(\Omega)$ is the subset of all functions from
$C^{\infty}(\Omega)$ with compact support in $\Omega$. For given
nonnegative  integers $k$ and $q > 1$, $L^q(\Omega)$ and $W^{k,
q}(\Omega)$ denote the usual Lebesgue and Sobolev spaces; $W^{k-1/q,q}(\partial \Omega)$ is the trace space on
$\partial\Omega$ of functions from $W^{k, q}(\Omega)$;
${W}^{k,q}_0(\Omega)$ is the closure of $C_{0}^{\infty}(\Omega)$ with respect to the
norm of $W^{k, q}(\Omega)$; if $\Omega$ is an unbounded domain, we write $u\in W^{k,q}_{
loc}(\overline{\Omega})$ if $u\in W^{k,q}(\Omega\cap B_R(0))$ for any $B_R(0)=\{x\in\sR^2:|x|\leq R\}.$

Let $D(\Omega)$ be the Hilbert space of vector valued functions formed as
the closure of $C_0^{\infty}(\Omega)$ with respect to the Dirichlet norm
$\|\textbf{u}\|_{ D(\Omega)}=\|\nabla\textbf{u}\|_{ L^2(\Omega)}$
induced by the scalar product
$$\begin{array}{l}
(\textbf{u},\textbf{v})=\int\limits_{\Omega}\nabla \textbf{u} :
\nabla\textbf{v} \ dx,
\end{array}
$$
where $\nabla\textbf{u} : \nabla\textbf{v}=\sum\limits^{n}_{j=1}
\nabla u_{j} \cdot \nabla v_{j}=\sum\limits_{j=1}^{n}
\sum\limits_{k=1}^{n} \dfrac{\partial u_{j}}{\partial x_{k}}
\dfrac{\partial v_{j}}{\partial x_{k}}.$ Denote  by
$J_0^{\infty}(\Omega)$  the set of all solenoidal ($\div {\bf u}=0$)
vector fields ${\bf u}$ from $C_0^{\infty}(\Omega)$. By $H(\Omega)$ we indicate the
space formed as the closure of $J_0^{\infty}(\Omega)$ with respect to the
Dirichlet norm. For any bounded domain $\Omega,$ $H^*(\Omega)$ denotes the dual of $H(\Omega).$
$|| \ \ ||_{H^*(\Omega)}$ denotes the strong dual norm in $H^*(\Omega).$

Assume that $\Omega$ is symmetric with respect to the $x_1$-axis, i.e.,
\begin{equation}\label{sd}
(x_1,x_2)\in\Omega \Leftrightarrow (x_1,-x_2)\in\Omega.
\end{equation}
The vector function ${\bf u}=(u_1, u_2)$ is called symmetric with respect to the $x_1$-axis if $u_1$ is an even function of $x_2$ and $u_2$ is an odd function of $x_2,$ i.e.
\begin{equation}\label{ss}
u_1(x_1,x_2)=u_1(x_1,-x_2), \ \ \ \ u_2(x_1,x_2)=-u_2(x_1,-x_2).
\end{equation}
\\For any set of functions $V(\Omega)$ defined in the symmetric domain $\Omega$ satisfying \eqref{sd}, we denote by $V_S(\Omega)$ the subspace of symmetric functions from $V(\Omega)$ satisfying \eqref{ss}.

Below we use the well known results which are formulated in the following two lemmas.

\begin{lem}\label{hardy}(see \cite{Lad})
Let $\Pi\subset\sR^2$ be a bounded domain with Lipschitz boundary $\partial\Pi.$ Then for any ${\bf w}\in W^{1,2}(\Pi)$ with ${\bf w}\big|_{\mathcal{L}}=0,$ $\mathcal{L}\subseteq\partial\Pi,$ $\meas (\mathcal{L})>0,$ the following inequality
\begin{equation}\label{hardy1}\begin{array}{l}
\int\limits_{\Pi}\dfrac{|{\bf w}|^2\,dx}{{\rm dist}^2(x,\mathcal{L})}\leq c\int\limits_{\Pi}|\nabla {\bf w}|^2\,dx
\end{array}
\end{equation}
holds.
\end{lem}

\begin{lem}\label{extension}(see \cite{Lad})
Let $\Pi\subset\sR^2$ be a bounded domain with Lipschitz boundary $\partial\Pi,$ $\mathcal{L}\subseteq\partial\Pi,$ $\meas (\mathcal{L})>0$ and ${\bf h}\in W^{1/2,2}(\partial\Pi)$ satisfying the conditions $\int\limits_{\mathcal{L}} {\bf h}\cdot {\bf n}\,dS=0,$ ${\rm supp}\,{\bf h}\subseteq \mathcal{L}.$ Then ${\bf h}$ can be extended inside $\Pi$ in the form
\begin{equation}\label{extension1}\begin{array}{l}
{\bf A}_*(x,\varepsilon)=\bigg(-\dfrac{\partial(\chi(x,\varepsilon)\, {\bf E}(x))}{\partial x_2}, \dfrac{\partial(\chi(x,\varepsilon)\, {\bf E}(x))}{\partial x_1}\bigg),
\end{array}
\end{equation}
where ${\bf E}\in W^{2,2}(\Pi),$ $ \bigg(-\dfrac{\partial {\bf E}(x)}{\partial x_2}, \dfrac{\partial{\bf E}(x)}{\partial x_1}\bigg)\bigg|_{\partial\Pi}={\bf h}$ and $\chi$ is a Hopf's type cut-off function, i.e. $\chi(x,\varepsilon)$ is smooth, $\chi(x,\varepsilon)=1$ on $\mathcal{L},$ ${\rm supp}\,\chi$ is contained in a small neighborhood of $\mathcal{L}$ and 
$$
|\nabla\chi(x,\varepsilon)|\leq \dfrac{\varepsilon\,c}{{\rm dist}(x,\mathcal{L})}.
$$
The constant $c$ is independent of $\varepsilon>0.$
\end{lem}

\setcounter{lem}{0}
\setcounter{equation}{0}
\section{Formulation of the Problem}

We study problem \eqref{prad0}
in a symmetric domain $\Omega\subset\sR^2$ having one outlet to infinity. Denote by $D^{(out)}$ the set
\begin{equation*}\begin{array}{l}
D^{(out)}=\{x\in\mathbb{R}^2:|x_2|\leq g(x_1), \ \ x_1>R_*>0 \},
\end{array}\end{equation*}
where $g=g(x_1)$ is a positive smooth function such that 
$g',$ $g\,g''$ are bounded on the interval $[R_*,\,+\infty)$
and satisfies the Lipschitz condition
$$
|g(t_1)-g(t_2)|\leq L\,|t_1-t_2|, \ \ \ t_1,t_2\geq R_*
$$
with the Lipschitz constant $L.$  
\\We call this set ``an outlet to infinity". Depending on the function $g$ the outlet $D^{(out)}$ can be paraboloidal type or channel-like outlet ($D^{(out)}$ is a channel-like outlet if $g(x_1)=const$).
\\Let us take a small positive number $\gamma$ and introduce another outlet
\begin{equation*}\begin{array}{l}
D^{(in)}=\{x\in D^{(out)}:|x_2|\leq\frac{\gamma}{\gamma+1}\,g(x_1), \ \ x_1>R_*>0 \}.
\end{array}\end{equation*}
\\We consider an unbounded symmetric domain 
$$\Omega = \Omega_0\cup D,$$
where $\Omega_0 = \Omega \cap B_{R_0} (0)$ is the bounded part of the domain $\Omega$ and the unbounded part $D$ is such that (see Fig.\,1) $$D^{(in)}\subset D\subset D^{(out)}.$$ 
\\We assume that

(i) the bounded domain $\Omega_0$ has the form
\begin{equation*}
\begin{array}{l}
\Omega_0=G_0\setminus \bigcup\limits_{i=1}^N \overline{G}_i,
\end{array}
\end{equation*}
where $G_0$ and $G_i,\ \ i=1,...,N,$ are
bounded simply connected domains such that $\overline{G}_i \subset
G_0$ and $G_N$ denotes the nearest ``hole" to the outlet. Denote $\partial G_i=\Gamma_i, \ i=1,...,N;$ 

(ii) the boundary $\partial\Omega$ is composed of the bounded connected components $\Gamma_i, \ i=1,..., N,$ and the infinite component
$\Gamma^*_0=\partial\Omega\setminus\bigcup\limits_{i=1}^N \Gamma_i.$ $\Gamma^*_0$ can be regarded as the outer boundary of $\Omega,$ while $\Gamma_i, \ \ i=1,...,N,$ as the inner boundaries. Denote $\Gamma_0=\Gamma^*_0\cap\partial\Omega_0.$ We suppose that each $\Gamma_i, \ i=0,...,N,$ intersects the $x_1$ axis.
\\Below we will use the following notation:
\begin{equation*}\begin{array}{l}
R_{l+1}=R_l+\dfrac{g(R_l)}{2L}, \ \ l\geq 0,\\
\\
D_l=\{x\in D: x_1<R_l \}, \ \ \Omega_l=\Omega_0\cup D_l, \ \ \ \omega_l=\Omega_{l+1}\setminus\overline{\Omega}_l.
\end{array}
\end{equation*}
\begin{figure}[ht!]
\centering
\includegraphics[scale=0.8]{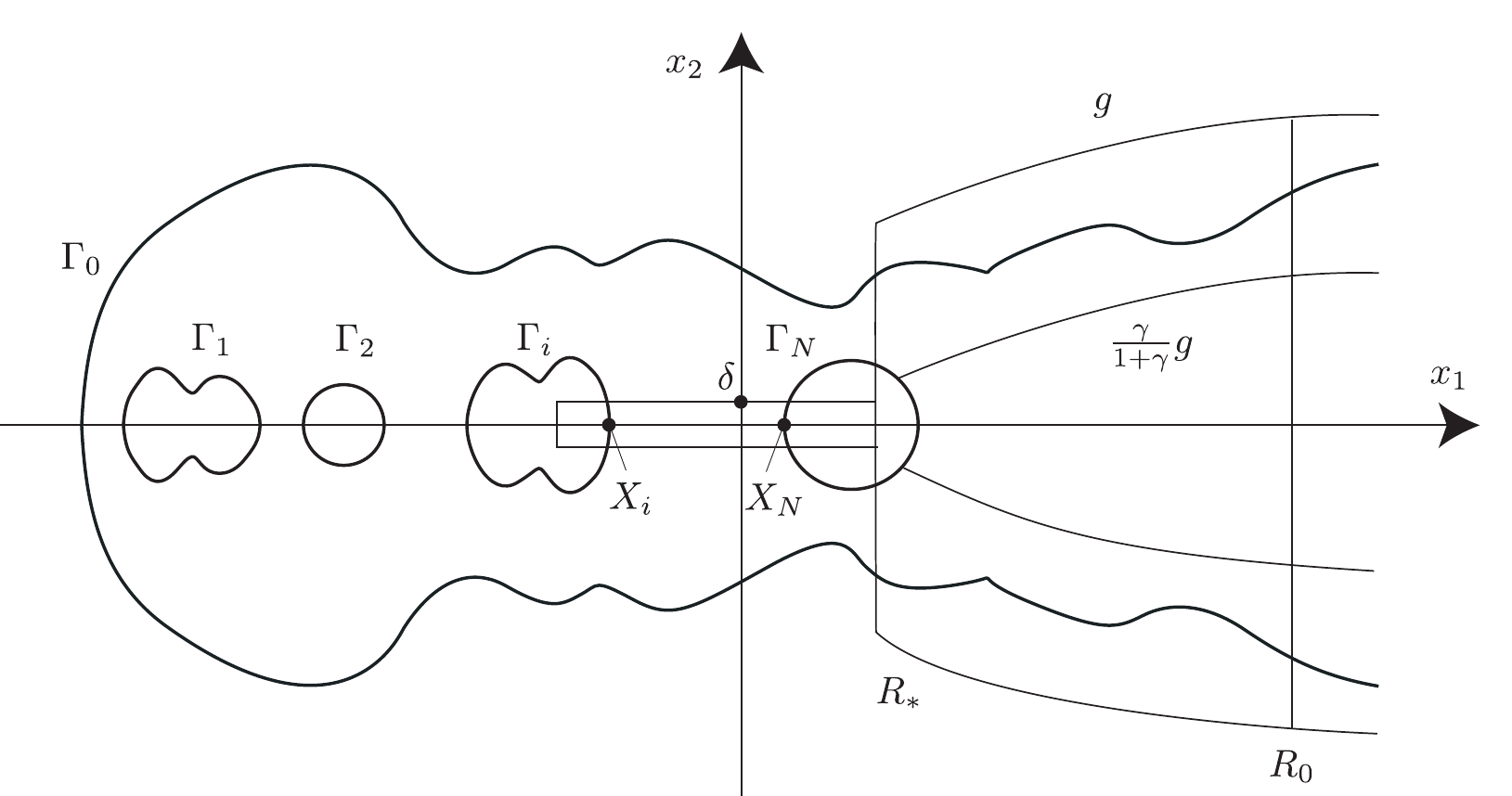}
\caption{The domain $\Omega$}
\end{figure}
\begin{remark}\label{remark1}
{\rm There holds the relation
\begin{equation}\label{g5}\begin{array}{l}
\dfrac{1}{2}g(R_k)\leq g(t)\leq \dfrac{3}{2}g(R_k), \ \ t\in [R_k, \ R_{k+1}].
\end{array}\end{equation}
Indeed, for $t\in [R_k, \ R_{k+1}]$ one has
\begin{equation*}\begin{array}{l}
-\dfrac{g(R_k)}{2}=-L\,\big(R_{k+1}-R_k\big)\leq -L\,\big(t-R_k \big)\leq g(t)-g(R_k)\\
\\
\leq L\,(t-R_k)\leq L\,(R_{k+1}-R_k)=L\,\dfrac{g(R_k)}{2L}.
\end{array}\end{equation*}
This implies \eqref{g5}.}
\end{remark}
We  suppose that the boundary value ${\bf a}\in
W^{1/2,2}(\partial\Omega)$ has a compact support.
Let
$$\begin{array}{l}
\int\limits_{\Gamma_i} {\bf a }\cdot {\bf n} \, dS=\mathbb{F}_i, \ i=0,...,N,
\end{array}$$
be the fluxes of the boundary value ${\bf a}$ over $\partial\Omega.$ 
Since the total flux has to be equal to zero (the necessary flux compatibility condition), we have
$$\begin{array}{l}\int\limits_{\sigma(R)} {\bf u}\cdot {\bf n}\,dS=-\sum\limits_{i=0}^N \mathbb{F}_i,\end{array} \ \ R>R_0>0,$$
where
$\sigma(R)$ is a cross section of the outlet $D.$

Two main purposes of this paper are: 

1) to construct a suitable symmetric
extension ${\bf A}$ of the boundary data ${\bf a}$
which satisfies the
Leray--Hopf type inequalities
\begin{equation}\label{ws4}\begin{array}{l}
\big|\int\limits_{\Omega_{k+1}}
({\bf w}\cdot\nabla){\bf w}\cdot{\bf A}\, dx\big| \leq
\varepsilon\int\limits_{\Omega_{k+1}}
|\nabla\textbf{w}|^2\, dx, \\
\\
\big|\int\limits_{\omega_k}
({\bf w}\cdot\nabla){\bf w}\cdot{\bf A} \, dx\big|\leq
\varepsilon\int\limits_{\omega_k}
|\nabla\textbf{w}|^2\, dx,
\end{array}\end{equation}
where ${\bf w}\in W^{1,2}_{loc}(\overline{\Omega})$ is an arbitrary solenoidal
function  with ${\bf w}\big|_{\partial\Omega}=0$ and $\varepsilon$ can be chosen arbitrary small;

2) to prove the existence of at least one weak symmetric solution ${\bf u}$ to problem \eqref{prad0}.

\setcounter{lem}{0}
\setcounter{remark}{0}
\setcounter{equation}{0}
\section{Construction of the Extension}

We construct a symmetric extension ${\bf A}$ of the boundary value ${\bf a}$ as a sum:
$${\bf A}={\bf B}_0+{\bf B}_{\infty}.$$
In order to construct an extension ${\bf B}_0$ we ``remove" the fluxes  $\mathbb{F}_i,\ \ i=0,...,N-1,$ to the boundary $\Gamma_N$ and then we extend the modified boundary value which has zero fluxes on $\Gamma_i, \ i=0,...,N-1,$ into $\Omega.$ After this step we get the flux $\sum\limits_{i=0}^N \mathbb{F}_i$ on $\Gamma_N.$ Then by removing it to infinity and extending the modified boundary value from $\Gamma_N$ into  $\Omega$ we construct the extension ${\bf B}_{\infty}.$ In general if the stringent outflow condition is not valid one cannot expect that there exists such an extension (see the counterexample in \cite{Tak}). However, under our symmetry assumptions such an extension can be constructed. The first part of the construction is inspired by some ideas of Fujita \cite{Fu1} and the second part - by techniques proposed in \cite{PilSol}.

\subsection{Construction of the Extension ${\bf B}_0.$}

Before we start to construct the extension ${\bf B}_0$ we introduce some auxiliary functions.
\\For $x\in D^{(out)}, \ x_2>0,$ we set (see \cite{PilSol})
\begin{equation}\label{xi0}\begin{array}{l}
\xi(x)=\xi(x_1,x_2)=\Psi\Big( \varepsilon\,\ln\dfrac{\gamma(g(x_1)-x_2)}{x_2} \Big),
\end{array}\end{equation}
where $0\leq\Psi\leq 1$ is a smooth monotone cut-off function:
\begin{equation*}\label{fi}
\begin{array}{l}
\Psi(t)=\begin{cases}
0, \ \ t\leq 0,\\
1, \ \ t\geq 1.
\end{cases}
\end{array}
\end{equation*}

\begin{lem}\label{xi}
$\xi$ is a smooth function vanishing near $x_2=g(x_1)$ and equal to $1$ in a neighbourhood of $x_2=0.$ Moreover it holds
\begin{equation}\label{xi1}\begin{array}{l}
\Big|\dfrac{\partial\xi}{\partial x_i}\Big|\leq\dfrac{c\,\varepsilon}{x_2}, \ \ \ i=1,2,
\end{array}\end{equation}

\begin{equation}\label{xi2}\begin{array}{l}
\Big|\dfrac{\partial\xi}{\partial x_i}\Big|\leq\dfrac{C(\varepsilon)}{g(x_1)}, \ \ \ \Big|\dfrac{\partial^2\xi}{\partial x_i\partial x_j}\Big|\leq\dfrac{C(\varepsilon)}{g^2(x_1)} \ \ \ i,j=1,2,
\end{array}\end{equation}
where $c$ is independent of $\varepsilon$ and $C(\varepsilon)$ denotes a constant depending on $\varepsilon.$
\end{lem}

\begin{proof}
First one notices that the support  of $\nabla\xi$ is contained in the set where
\begin{equation}\label{xi3}\begin{array}{l}
1\leq\dfrac{\gamma\big(g(x_1)-x_2 \big)}{x_2}\leq e^{1/{\varepsilon}} \ \Leftrightarrow \ \ \dfrac{\big(1+\gamma\big)\,x_2}{\gamma}\leq g(x_1)\leq\dfrac{\big(e^{1/{\varepsilon}}+\gamma\big)\,x_2}{\gamma}.
\end{array}\end{equation}
Then since $\xi(x)=\Psi\big(\varepsilon\,\ln\big(g(x_1)-x_2\big)-\varepsilon\,\ln\,x_2+\varepsilon\,\ln\,\gamma \big)$ one gets
\begin{equation}\label{xi4}\begin{array}{l}
\dfrac{\partial\xi}{\partial x_1}=\Psi'\cdot\dfrac{\varepsilon\,g'(x_1)}{g(x_1)-x_2}, \ \ \ 
\dfrac{\partial\xi}{\partial x_2}=\Psi'\cdot\varepsilon\Big(\dfrac{-1}{g(x_1)-x_2}-\dfrac{1}{x_2}\Big),
\end{array}\end{equation}
where $\Psi'$ is taken at the point $\varepsilon\,\ln\dfrac{\gamma\,\big(g(x_1)-x_2 \big)}{x_2}$. Since we assumed that $g'$ is bounded and $\Psi'$ is bounded as well one derives from \eqref{xi3}, \eqref{xi4}
\begin{equation*}\begin{array}{l}
\Big|\dfrac{\partial\xi}{\partial x_1}\Big|\leq\dfrac{c\,\varepsilon}{x_2}, \ \ \ 
\Big|\dfrac{\partial\xi}{\partial x_2}\Big|\leq c\,\varepsilon\,\Big(\dfrac{1}{g(x_1)-x_2}+\dfrac{1}{x_2}\Big)\leq\dfrac{c\,\varepsilon}{x_2}
\end{array}\end{equation*}
for some constant $c$ and
$$\begin{array}{l}
\Big|\dfrac{\partial\xi}{\partial x_i}\Big|\leq\dfrac{c\,\varepsilon}{x_2}\leq\dfrac{c\,\varepsilon\big(\gamma+e^{1/{\varepsilon}}\big)}{\gamma\,g(x_1)}, \ \ \ i=1,2.
\end{array}$$
Thus, it remains only to prove the last inequality of \eqref{xi2}. Differentiating \eqref{xi4} we get 
\begin{equation*}\begin{array}{l}
\dfrac{\partial^2\xi}{\partial x_1^2}=\Psi''\Big(\dfrac{\varepsilon\,g'(x_1)}{g(x_1)-x_2} \Big)^2+\Psi'\varepsilon\Big(\dfrac{g''\big(g(x_1)-x_2\big)-g'\,^2}{\big(g(x_1)-x_2\big)^2} \Big).
\end{array}\end{equation*}
Since we assumed that $g''\,g$ is bounded, by \eqref{xi3} we obtain
\begin{equation*}\begin{array}{l}
\Big|\dfrac{\partial^2\xi}{\partial x_1^2} \Big|\leq\dfrac{c\,\varepsilon}{\big(g(x_1)-x_2\big)^2}\leq\dfrac{c\,\varepsilon}{x_2^2}\leq\dfrac{c\,\varepsilon\big(\dfrac{\gamma+e^{1/\varepsilon}}{\gamma} \big)^2}{g^2(x_1)}=\dfrac{C(\varepsilon)}{g^2(x_1)}.
\end{array}\end{equation*}
Similarly we have
\begin{equation*}\begin{array}{rcl}
\Big|\dfrac{\partial^2\xi}{\partial x_1\partial x_2} \Big|&=&\Big|\Psi''\varepsilon^2\dfrac{g'(x_1)}{g(x_1)-x_2}\cdot\Big(-\dfrac{1}{g(x_1)-x_2}-\dfrac{1}{x_2}\Big)+\Psi'\varepsilon\dfrac{g'(x_1)}{\big(g(x_1)-x_2\big)^2}\Big|\\
\\
&\leq &\dfrac{c\,\varepsilon}{x_2^2}\leq\dfrac{C(\varepsilon)}{g^2(x_1)},\end{array}\end{equation*}
\begin{equation*}\begin{array}{rcl}
\Big|\dfrac{\partial^2\xi}{\partial x_2^2} \Big|&=&\Big|\Psi''\varepsilon^2\Big(-\dfrac{1}{g(x_1)-x_2}-\dfrac{1}{x_2}\Big)^2+\Psi'\varepsilon\,\Big(-\dfrac{1}{\big(g(x_1)-x_2 \big)^2}+\dfrac{1}{x_2^2}\Big)\Big|\\
\\
&\leq & \dfrac{c\,\varepsilon}{x_2^2}\leq\dfrac{C(\varepsilon)}{g^2(x_1)}.
\end{array}\end{equation*}
This completes the proof of the Lemma.
\end{proof}
We set 
\begin{equation}\label{zeta1}\begin{array}{l}
\tilde{\bxi}(x)=(\tilde{\xi}_1,  \tilde{\xi}_2)=\begin{cases}
\Big(-\dfrac{\partial\xi(x_1,x_2)}{\partial x_2}, \  \dfrac{\partial\xi(x_1,x_2)}{\partial x_1} \Big), \ \ x_2> 0,\\
\\
\Big(-\dfrac{\partial\xi(x_1,-x_2)}{\partial x_2}, \  -\dfrac{\partial\xi(x_1,-x_2)}{\partial x_1} \Big), \ \ x_2< 0.
\end{cases}
\end{array}\end{equation}
Then we have
\begin{lem}\label{tilde}
$\tilde{\bxi}$ is a smooth solenoidal symmetric vector field such that for any cross section $\sigma^{(out)}$ of $D^{(out)}$ one has
\begin{equation}\label{xi5}\begin{array}{l}
\int\limits_{\sigma^{(out)}}\tilde{\bxi}\cdot {\bf e}_1\,dx_2=2.
\end{array}\end{equation}
\end{lem}

\begin{proof}
Since function $\dfrac{\partial \xi}{\partial x_2}$ is even in $x_2$ we obtain
\begin{equation*}\begin{array}{l}
\int\limits_{\sigma^{(out)}}\tilde{\bxi}\cdot {\bf e}_1\,dx_2=2\int\limits_0^{g(x_1)}\big(-\dfrac{\partial \xi}{\partial x_2} \big)\,dx_2=-2\,\xi(x_1,g(x_1))+2\xi(x_1,0)=2.
\end{array}\end{equation*}
\end{proof}

\begin{figure}[ht!]
\centering
\includegraphics[scale=0.8]{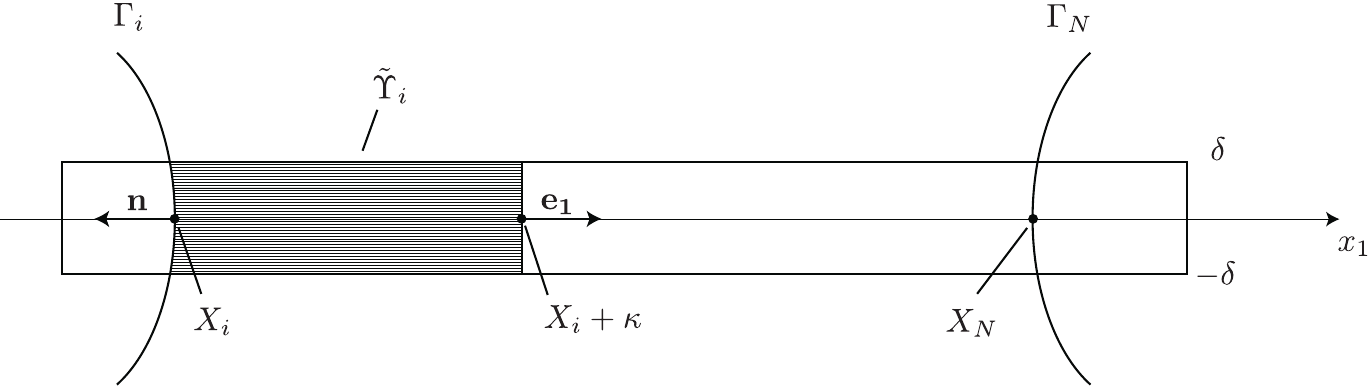}
\caption{The strip $\Upsilon_i$}
\end{figure}
Now we start to construct the extension ${\bf B}_0.$ Let us choose $\delta$ small enough in such a way that the straight line $x_2=\delta$ cuts each of the $\Gamma_i, \ \ i=1,...,N,$ at only two points. For $i=0,...,N-1$ we define the thin strips $\Upsilon_i=[X_i-\eta_i, \, X_N+\eta_N]\times [-\delta,\,\delta],$ where $\eta_i$ and $\eta_N$ are small positive numbers and note that the points $(X_i-\eta_i,\,0)$ and $(X_N+\eta_N,\,0)$ are outside of the domain $\Omega$ (see Fig.\,1).
Then on each strip $\Upsilon_i\cap\Omega, \ \ i=0,...,N-1,$ joining $\Gamma_i$ to $\Gamma_N$ we define ${\bf b}_i$ in the following way 
\begin{equation*}\label{xi7}
\begin{array}{l}
{\bf b}_i (x)=-\mathbb{F}_i\,(\tilde{\xi}_1,  \tilde{\xi}_2)=
\begin{cases}
\dfrac{\mathbb{F}_i}{2}\Big(\dfrac{\partial\xi_{\delta}(x_2)}{\partial x_2},\,0 \Big), \ \ \ {\rm in} \ \Upsilon_i\cap\Omega, \ \ x_2>0,\\
\\
\dfrac{\mathbb{F}_i}{2}\Big(\dfrac{\partial\xi_{\delta}(-x_2)}{\partial x_2},\,0 \Big), \ \ \ {\rm in} \ \Upsilon_i\cap\Omega, \ \ x_2<0,
\\
\\
(0,\,0), \ \ {\rm in} \ \overline{\Omega}\setminus(\Upsilon_i\cap\Omega),
\end{cases}
\end{array}
\end{equation*}
where
\begin{equation*}\begin{array}{l}
\xi_{\delta}(x)=\Psi\Big(\varepsilon\,\ln\dfrac{\delta-x_2}{x_2}\Big).
\end{array}\end{equation*}
Notice that the Lemma\,\ref{xi} and Lemma\,\ref{tilde} are valid if we take $\gamma=1$ and $g(x_1)=\delta.$ 
\\Since each vector field ${\bf b}_i$ is solenoidal and vanishes on the upper and lower boundaries of $\Upsilon_i$, we have
$$\begin{array}{l}
0=\int\limits_{\tilde{\Upsilon}_i} {\rm div}\,{\bf b}_i\,dx=\int\limits_{\partial\tilde{\Upsilon}_i}{\bf b}_i\cdot \nn \,dS
=\int\limits_{\Gamma_i}{\bf b}_i\cdot {\bf n}\, dS+\int\limits_{(X_i+\kappa)\times[-\delta,\delta]}{\bf b}_i\cdot {\bf e}_1 \,dS\\
\\
=\int\limits_{\Gamma_i}{\bf b}_i\cdot {\bf n}\, dS-
\dfrac{\mathbb{F}_i}{2}\cdot2\,\int\limits_{0}^{\delta} \big(-\dfrac{\partial\xi_{\delta}}{\partial x_2}\big)\,dx_2
=\int\limits_{\Gamma_i}{\bf b}_i\cdot {\bf n} \,dS-\mathbb{F}_i, \ \ \ \forall i=0,...,N-1,
\end{array}$$
where $\tilde{\Upsilon}_i$ is the domain enclosed by $\Gamma_i,$ $(X_i+\kappa)\times[-\delta, \delta]$ ($\kappa$ is a small positive number) and the lines $x_2=\delta,$ $x_2=-\delta$ (see Fig.\,2).
Therefore, it follows that if ${\bf n}$ denotes the unit outward normal to $\partial\Omega$ on $\Gamma_i$ one has
$$\begin{array}{l}
\int\limits_{\Gamma_i}{\bf b}_i\cdot \nn \,dS=\mathbb{F}_i, \ \ \ \forall i=0,...,N-1.
\end{array}$$
We set
$$\mathbf{b}=\sum\limits_{i=0}^{N-1}\mathbf{b}_i.$$
Clearly ${\bf b}$ is a symmetric solenoidal vector field. Moreover for every $i=0,...,N-1$ one has (note that the flux of ${\bf b}_i$ vanishes on $\Gamma_j$ for every $i\neq j$)
\begin{equation}\label{zero}\begin{array}{l}
\int\limits_{\Gamma_i}({\bf a}-{\bf b})\cdot {\bf n}\,dS=\int\limits_{\Gamma_i}({\bf a}-{\bf b}_i)\cdot {\bf n}\,dS=\mathbb{F}_i-\mathbb{F}_i=0.
\end{array}\end{equation}
Because of \eqref{zero} there exists (see Lemma\,\ref{extension}) an extension ${\bf A}_0$ of $\big({\bf a}-{\bf b}\big)\Big|_{\bigcup\limits_{i=0}^{N-1}\Gamma_i}$ such that ${\rm supp}\,{\bf A}_0$ is contained in a small neighborhood of $\bigcup\limits_{i=0}^{N-1}\Gamma_i,$
\begin{equation*}\label{b0}\begin{array}{l}
{\rm div}\,{\bf A}_0=0, \ \ {\bf A}_0|_{\bigcup\limits_{i=0}^{N-1}\Gamma_i}=\big({\bf a}-{\bf b}\big)\Big|_{\bigcup\limits_{i=0}^{N-1}\Gamma_i},
\end{array}
\end{equation*}
and ${\bf A}_0$ satisfies the Leray--Hopf inequalities\footnote{Notice that the integral over $\omega_k$ is equal to zero since ${\bf A}_0=0$ in $\omega_k.$} for every solenoidal function ${\bf w} \in W^{1,2}_{loc}(\overline{\Omega})$ with ${\bf w}|_{\partial\Omega}=0$
\begin{equation}\label{binn0}\begin{array}{l}
\big|\int\limits_{\Omega_{k+1}}({\bf w} \cdot \nabla){\bf w} \cdot {\bf A}_0\, dx\big| \leq c\,\varepsilon\int\limits_{\Omega_{k+1}}|\nabla {\bf w}|^2\,dx.
\end{array}\end{equation}
\\Notice that the vector field ${\bf A}_0$ is not necessary symmetric. However, since the boundary value $\big({\bf a}-{\bf b}\big)\Big|_{\bigcup\limits_{i=0}^{N-1}\Gamma_i}$ is symmetric, ${\bf A}_0$ can be symmetrized to $\tilde{\bf A}_0$ where for ${\bf A}= ({A}_{1},{A}_{2})$ we define $\tilde{\bf A}= (\tilde{A}_{1},\tilde{A}_{2})$ as follows:
\begin{equation}\label{sym}\begin{array}{l}
\tilde{A}_{1}(x)= \dfrac{1}{2}\bigg(A_{1} (x_1,x_2) +  A_{1} (x_1,-x_2)\bigg),\ \ x\in\Omega,\\
\\
\tilde{A}_{2}(x)=\dfrac{1}{2}\bigg(A_{2} (x_1,x_2) -  A_{2} (x_1,-x_2)\bigg), \ \ x\in\Omega.
\end{array}
\end{equation}
Define then
$$\mathbf{B}_0=\tilde{\mathbf{A}}_0+\mathbf{b}.$$
Hence, ${\bf B}_0$ is a symmetric extension of the boundary value ${\bf a}$ from $\bigcup\limits_{i=0}^{N-1}\Gamma_i$. Moreover
\begin{equation*}\begin{array}{l}
\int\limits_{\Gamma_N}({\bf a}-{\bf b})\cdot {\bf n}\,dS= \sum\limits_{i=0}^{N} \mathbb{F}_i=\mathbb{F}.
\end{array}\end{equation*}
It remains to prove that ${\bf B}_0$ satisfies the Leray-Hopf inequalities. It is enough to prove that each ${\bf b}_i$, $i=0,\cdots, N-1,$ satisfies the Leray-Hopf inequalities.

Let ${\bf w} =(w_1,w_2)\in W^{1,2}_{loc}(\overline{\Omega}), \ {\bf w}|_{\partial\Omega}=0,$ be a symmetric and solenoidal vector field. Then
\begin{equation*}\begin{array}{l}
\int\limits_{\Omega_{k+1}}({\bf w} \cdot \nabla)\,{\bf w}\cdot{\bf b}_i\,dx=-\F_i\int\limits_{\Upsilon_i\cap\overline{\Omega}}\big(w_1 \dfrac{\partial w_1}{\partial x_1}+w_2\dfrac{\partial w_1}{\partial x_2}\big)\,\tilde{\xi}_1(x_2)\,dx
\end{array}\end{equation*}
Since $\dfrac{w_1^2(x)}{2}\, \tilde{\xi}_1(x_2)$ vanishes on the boundary of $\Upsilon_i\cap \overline \Omega$ one has
\begin{equation*}\begin{array}{l}
\int\limits_{\Upsilon_i\cap\overline{\Omega}}w_1(x) \dfrac{\partial w_1(x)}{\partial x_1}\, \tilde{\xi}_1(x_2)\,dx=\dfrac{1}{2}\int\limits_{\Upsilon_i\cap\overline{\Omega}} \dfrac{\partial\big(w_1^2(x)\, \tilde{\xi}_1(x_2)\big)}{\partial x_1}\,dx=0. 
\end{array}\end{equation*}
Therefore, using the definition of $\tilde{\xi}_1,$ applying the estimate \eqref{xi1} and the Hardy\footnote{For the application of the Hardy type inequality we used the fact that $w_2$ vanishes on $x_2=0$.} type inequality one gets
\begin{equation*}\begin{array}{l}
\big|\int\limits_{\Omega_{k+1}}({\bf w} \cdot \nabla){\bf w}\cdot{\bf b}_i\,dx\big|
\leq  |\F_i|\int\limits_{\Upsilon_i\cap\Omega}\Big|w_2\,\dfrac{\partial w_1}{\partial x_2}\,\tilde{\xi}_1 \Big|\,dx\\
\\
\leq  c\,\varepsilon|\F_i|\int\limits_{\Upsilon_i\cap\Omega} \dfrac{|w_2|}{|x_2|}\,\big|\dfrac{\partial w_1}{\partial x_2} \big| \,dx
\leq  c\,\varepsilon|\F_i|\int\limits_{\Omega_{k+1}}|\nabla {\bf w}|^2\,dx.
\end{array}\end{equation*}
\\Thus, we have proved the following lemma.
\begin{lem}\label{binnl} Assume that the boundary value ${\bf a}$ is a symmetric function in $W^{1/2,2}(\partial\Omega)$ having a compact support. Denote by $\tilde{{\bf a}}$ the restriction of ${\bf a}$ to $\bigcup\limits_{i=0}^{N-1}\Gamma_i.$ Then for every $\varepsilon > 0$ there exists a symmetric solenoidal extension ${\bf B}_0$ in $\Omega$ satisfying ${\bf B}_0\big|_{\bigcup\limits_{i=0}^{N-1}\Gamma_i}=\tilde{{\bf a}},$ ${\bf B}_0\big|_{\partial\Omega\setminus\bigcup\limits_{i=0}^{N-1}\Gamma_i}=0$  and the Leray-Hopf inequalities\footnote{Notice that the integral over $\omega_k$ is equal to zero since ${\bf B}_0=0$ in $\omega_k.$}, i.e., for every symmetric solenoidal function ${\bf w} \in W^{1,2}_{loc}(\overline{\Omega})$ with ${\bf w}|_{\partial\Omega}=0$ the following estimates
\begin{equation}\label{binn}\begin{array}{l}
\big|\int\limits_{\Omega_{k+1}}({\bf w} \cdot \nabla){\bf w} \cdot {\bf B}_0\, dx\big| \leq c\,\varepsilon\, \int\limits_{\Omega_{k+1}}|\nabla {\bf w}|^2\,dx
\end{array}\end{equation}
hold.
\end{lem}

\begin{remark}
{\rm In the case of a bounded domain the vector field ${\bf B}_0$ is a suitable extension of the boundary value ${\bf a},$ i.e. ${\bf A}={\bf B}_0.$ The idea of the construction of ${\bf B}_0$ is very similar to that of H. Fujita (\cite{Fu1}).}
\end{remark}

\subsection{Construction of the Extension ${\bf B}_{\infty}.$}

After moving all the fluxes through $\Gamma_i, \ \ i=0,...,N-1,$ to the last inner boundary $\Gamma_N$ we need to drain the flux from $\Gamma_N$ to infinity. 
There we consider a function $g$ as in Lemma\,\ref{xi} and suppose that $\gamma$ is chosen such that
\begin{equation}\label{xi6}\begin{array}{l}
{\rm the} \ {\rm curve} \ \ x_2=\dfrac{\gamma}{\gamma+1}\,g(x_1) \ \ {\rm crosses} \ \ \Gamma_N.
\end{array}\end{equation}
Let us introduce the vector field
\begin{equation*}\begin{array}{l}
{\bf b}_{\infty}(x)=-\mathbb{F}\,\tilde{\bxi}=\begin{cases}
\dfrac{\mathbb{F}}{2}\,\Big(\dfrac{\partial\xi(x_1,x_2)}{\partial x_2}, \  -\dfrac{\partial\xi(x_1,x_2)}{\partial x_1} \Big), \ \ x_2> 0,\\
\\
\dfrac{\mathbb{F}}{2}\,\Big(\dfrac{\partial\xi(x_1,-x_2)}{\partial x_2}, \  \dfrac{\partial\xi(x_1,-x_2)}{\partial x_1} \Big), \ \ x_2< 0,
\end{cases}
\end{array}\end{equation*}
where $\xi$ is defined by \eqref{xi0} for $x\in D^{(in)}$ and extended by $0$ into $D.$ Then since for any cross section $\sigma$
\begin{equation*}\begin{array}{l}
\int\limits_{\Gamma_N}{\bf b}_{\infty}\cdot {\bf n}\,dS=-\int\limits_{\sigma}{\bf b}_{\infty}\cdot {\bf n}\,dS=
-\dfrac{\mathbb{F}}{2}\cdot 2\int\limits_0^{\frac{\gamma}{\gamma+1}\,g(x_1)}\dfrac{\partial\xi}{\partial x_2}\,dx_2=\\
\\
=-\mathbb{F}\,\Big(\xi(x_1,\,\frac{\gamma}{\gamma+1}\,g(x_1))-\xi(x_1,\,0) \Big)=\mathbb{F},
\end{array}\end{equation*}
one has
\begin{equation}\label{inff}\begin{array}{l}
\int\limits_{\Gamma_N}\big({\bf a}-{\bf b}-{\bf b}_{\infty} \big)\,{\bf n}\,dS=0.
\end{array}\end{equation}
Because of \eqref{inff} there exists (see Lemma\,\ref{extension}) an extension ${\bf A}_{\infty}$ of $\big({\bf a}-{\bf b}-{\bf b}_{\infty}\big)\Big|_{\Gamma_N}$ such that ${\rm supp}\,{\bf A}_{\infty}$ is contained in a small neighborhood of $\Gamma_N,$
\begin{equation*}\label{b0}\begin{array}{l}
{\rm div}\,{\bf A}_{\infty}=0, \ \ {\bf A}_{\infty}|_{\Gamma_N}=\big({\bf a}-{\bf b}-{\bf b}_{\infty}\big)\Big|_{\Gamma_N},
\end{array}
\end{equation*}
and ${\bf A}_{\infty}$ satisfies the Leray--Hopf inequalities for every solenoidal function ${\bf w} \in W^{1,2}_{loc}(\overline{\Omega})$ with ${\bf w}|_{\partial\Omega}=0$
\begin{equation}\label{binn0}\begin{array}{l}
\big|\int\limits_{\Omega_{k+1}}({\bf w} \cdot \nabla){\bf w} \cdot {\bf A}_{\infty}\, dx\big| \leq c\,\varepsilon \int\limits_{\Omega_{k+1}}|\nabla {\bf w}|^2\,dx,\\
\\
\big|\int\limits_{\omega_{k}}({\bf w} \cdot \nabla){\bf w} \cdot {\bf A}_{\infty}\, dx\big| \leq c\,\varepsilon \int\limits_{\omega_{k}}|\nabla {\bf w}|^2\,dx
\end{array}\end{equation}
with a constant $c$ independent of $k$ and $\varepsilon.$
Notice that the vector field ${\bf A}_{\infty}$ is not necessary symmetric. However, since the boundary value $\big({\bf a}-{\bf b}-{\bf b}_{\infty}\big)\Big|_{\Gamma_N}$ is symmetric, ${\bf A}_{\infty}$ can be symmetrized to $\tilde{\bf A}_{\infty}$ as in \eqref{sym}. Then
\begin{equation*}\begin{array}{l}
{\bf B}_{\infty}={\bf b}+{\bf b}_{\infty}+\tilde{\bf A}_{\infty}
\end{array}\end{equation*}
is a symmetric solenoidal extension of ${\bf a}$ on $\Gamma_N.$ It remains to prove that ${\bf B}_{\infty}$ satisfies the Leray-Hopf inequalities. It is enough to prove that ${\bf b}_{\infty}$ satisfies the Leray-Hopf inequalities.

Let ${\bf w} =(w_1,w_2)\in W^{1,2}_{loc}(\overline{\Omega}), \ {\bf w}|_{\partial\Omega}=0,$ be a symmetric and solenoidal vector field.
We use the well known identity
\begin{equation}\label{wki}\begin{array}{l}
({\bf w}\cdot\nabla)\,{\bf w}=\nabla(\dfrac{1}{2}|{\bf w}|^2)+\big(\dfrac{\partial w_2}{\partial x_1}-\dfrac{\partial w_1}{\partial x_2} \big)\,(-w_2,\,w_1).
\end{array}\end{equation}
Since ${\bf b}_{\infty}$ is solenoidal, it is $L^2-$ orthogonal to the first term of the right-hand side of \eqref{wki}. Then one obtains
\begin{equation}\label{last1}\begin{array}{l}
\Big|\int\limits_{\Omega_{k+1}} ({\bf w}\cdot\nabla)\,{\bf w}\,\cdot {\bf b}_{\infty}\,dx\Big|
\leq |\mathbb{F}|\,\int\limits_{\Omega_{k+1}}\Big|\big(\dfrac{\partial w_2}{\partial x_1}-\dfrac{\partial w_1}{\partial x_2} \big)\,(-w_2\,\tilde{\xi}_1+w_1\,\tilde{\xi}_2)\Big|\,dx\\
\\
\leq |\mathbb{F}|\,\Big(\int\limits_{\Omega_{k+1}}|\nabla {\bf w}|^2dx\Big)^{1/2}\Big(\int\limits_{\Omega_{k+1}}| (-w_2\,\tilde{\xi}_1+w_1\,\tilde{\xi}_2) |^2dx\Big)^{1/2}.
\end{array}\end{equation}
Let $G^{\pm}$ denotes the curve $x_2=\pm g(x_1).$ Then using \eqref{xi1}, \eqref{xi4} for $x\in\Omega$ and $x_2>0,$ we have
\begin{equation}\label{last2}\begin{array}{l}
|\tilde{\bxi}_1|=\Big|\dfrac{\partial\xi}{\partial x_2}\Big|\leq c\varepsilon\dfrac{1}{x_2}, \ \ \ \ \ |\tilde{\bxi}_2|=\Big|\dfrac{\partial\xi}{\partial x_1}\Big|\leq\dfrac{c\,\varepsilon}{{\rm dist}(x,G^+)}.
\end{array}\end{equation}
Therefore, from \eqref{last1} and \eqref{last2} applying\footnote{Here we used the fact that $w_2=0$ on $x_2=0$ and we supposed that ${\bf w}$ is extended by $0$ outside $\Omega.$} Hardy type inequality (see Lemma\,\ref{hardy}) we get
\begin{equation*}\begin{array}{l}
\Big|\int\limits_{\Omega_{k+1}^+} ({\bf w}\cdot\nabla)\,{\bf w}\,\cdot {\bf b}_{\infty}\,dx\Big|\\
\\
\leq  c\,\varepsilon\,|\mathbb{F}|\Big(\int\limits_{\Omega_{k+1}^+}|\nabla {\bf w}|^2dx\Big)^{1/2}
\Big[\Big(\int\limits_{\Omega_{k+1}^+}\,\dfrac{|w_2|^2}{|x_2|^2}\,dx\Big)^{1/2}+ \Big(\int\limits_{\Omega_{k+1}^+}\,\dfrac{|w_1|^2}{{\rm dist}^2(x,G^+)}\,dx\Big)^{1/2}\Big]
\\
\\
\leq  c\,\varepsilon\,|\mathbb{F}|\int\limits_{\Omega_{k+1}^+}|\nabla {\bf w}|^2\,dx,
\end{array}\end{equation*}
where $\Omega_{k+1}^+=\{x\in\Omega_{k+1}: \ x_2>0 \}.$
The same estimate is valid in $\Omega_{k+1}^-.$ Therefore, ${\bf b}_{\infty}$ satisfies the Leray-Hopf inequalities for every symmetric solenoidal function ${\bf w} \in W^{1,2}_{loc}(\overline{\Omega})$ with ${\bf w}|_{\partial\Omega}=0$
\begin{equation}\label{inf}\begin{array}{l}
\big|\int\limits_{\Omega_{k+1}}({\bf w} \cdot \nabla){\bf w} \cdot {\bf b}_{\infty}\, dx\big| \leq c\,\varepsilon\, \int\limits_{\Omega_{k+1}}|\nabla {\bf w}|^2\,dx,\\
\\
\big|\int\limits_{\omega_{k}}({\bf w} \cdot \nabla){\bf w} \cdot {\bf b}_{\infty}\, dx\big| \leq c\,\varepsilon\, \int\limits_{\omega_{k}}|\nabla {\bf w}|^2\,dx.
\end{array}\end{equation}
Moreover, one has the estimates
\begin{equation}\label{inf1}\begin{array}{l}
|{\bf b}_{\infty}|\leq\dfrac{C(\varepsilon)}{g(x_1)}, \ \ \ \ |\nabla{\bf b}_{\infty}|\leq\dfrac{C(\varepsilon)}{g^2(x_1)}, \ \ \ x\in D.
\end{array}\end{equation}
Hence together with Lemma\,\ref{binnl} we proved the following result.
\begin{lem}\label{infl} Assume that the boundary value ${\bf a}$ is a symmetric function in $W^{1/2,2}(\partial\Omega)$ having a compact support. Then for every $\varepsilon > 0$ there exists a symmetric solenoidal extension ${\bf A}={\bf B}_0+{\bf B}_{\infty}$ in $\Omega$ satisfying the Leray-Hopf inequalities, i.e., for every symmetric solenoidal function ${\bf w} \in W^{1,2}_{loc}(\overline{\Omega})$ with ${\bf w}|_{\partial\Omega}=0$ the following estimates
\begin{equation}\label{A}\begin{array}{l}
\big|\int\limits_{\Omega_{k+1}}({\bf w} \cdot \nabla){\bf w} \cdot {\bf A}\, dx\big| \leq c\,\varepsilon\,\int\limits_{\Omega_{k+1}}|\nabla {\bf w}|^2\,dx,\\
\\
\big|\int\limits_{\omega_{k}}({\bf w} \cdot \nabla){\bf w} \cdot {\bf A}\, dx\big| \leq c\,\varepsilon\,\int\limits_{\omega_{k}}|\nabla {\bf w}|^2\,dx
\end{array}\end{equation}
hold. The constant $c$ is independent of $k$ and $\varepsilon.$
\end{lem}
\begin{remark}
{\rm The constant $c$ in \eqref{A} is of the type 
\begin{equation*}\begin{array}{l}
c_1\sum\limits_{i=0}^N |\mathbb{F}_i|=c_1\sum\limits_{i=0}^N \big|\int\limits_{\Gamma_i} {\bf a}\cdot {\bf n}\,dS\big|\leq c_2\,\|{\bf a} \|_{W^{1/2,\,2}(\partial\Omega)},
\end{array}\end{equation*}
where $c_2$ is independent of ${\bf a}.$
 }
\end{remark}

\setcounter{lem}{0}
\setcounter{remark}{0}
\setcounter{equation}{0}
\section{Existence Theorem}

We look for the solution $\textbf{u}$ in the form
\begin{equation}\label{ws}
\textbf{u}(x)=\textbf{A}(x,\varepsilon)+\textbf{v}(x),\end{equation}
where ${\bf A}$ is the symmetric extension of the boundary value ${\bf a}$ constructed in the previous section (see Lemma\,\ref{infl}).

\begin{defi}\label{d}
Under a weak solution of problem \eqref{prad0} we understand a solenoidal vector field $\textbf{u}$ which is of the type \eqref{ws} with the symmetric vector field ${\bf v}\in W_{loc}^{1,2} (\overline{\Omega})$ satisfying the following conditions
$$\begin{array}{l}
\div {\bf  v}=0 \ {\rm in} \ \Omega, \quad {\bf v}=0 \ {\rm on} \ \partial\Omega,
\end{array}
$$
and the integral identity
\begin{equation}\label{ws3}\begin{array}{l}
\nu\int\limits_{\Omega} \nabla {\bf v} : \nabla\bfeta \, dx -
\int\limits_{\Omega} (({\bf A}+{\bf v})\cdot\nabla)\bfeta\cdot{\bf v} \, dx
-\int\limits_{\Omega} ({\bf v}\cdot\nabla)\bfeta\cdot{\bf A} \,
dx\\
\\
=\int\limits_{\Omega} ({\bf A}\cdot\nabla)\bfeta\cdot{\bf A} \, dx -
\nu\int\limits_{\Omega} \nabla {\bf A} : \nabla\bfeta \, dx+\int\limits_{\Omega}{\bf f}\cdot \bfeta\,dx \ \ \ \
\ \forall \bfeta\in J^\infty_{0}(\Omega).
\end{array}\end{equation}
\end{defi}
Then we have 
\begin{theo}\label{th}
Suppose that $\Omega\subset\sR^2$ is an unbounded domain symmetric with respect to the $x_1$ axis and each $\Gamma_i, \ \ i=0,...,N,$ intersects the $x_1$ axis.  Assume that the boundary value $\textbf{a}$ is a symmetric field in $W^{1/2,2}(\partial\Omega)$ having a compact support. Let ${\bf f}$ be a distribution which is symmetric in the sense that
\begin{equation}\label{ch1}\begin{array}{l}
<{\bf f}, \bfeta>=<{\bf f}, \tilde{\bfeta}> \ \ \forall\bfeta\in J^{\infty}_0(\Omega)
\end{array}\end{equation}
($\tilde{\bfeta}$ denotes the symmetrization of $\bfeta$ as defined in \eqref{sym} for ${\bf A}$) and such that
\begin{equation*}\begin{array}{l}
{\bf f}\in H^*(\Omega_k) \ \ \forall k \ \ {\rm and} \ \
\|{\bf f}\|_*=\sup\limits_{k\geq 1}\bigg(\Big(1+\int\limits_{R_0}^{R_k}\dfrac{dx_1}{g^3(x_1)}\Big)^{-1/2}\,\|{\bf f}\|_{H^{*}(\Omega_k)}\bigg)<+\infty.
\end{array}\end{equation*}
If either
\begin{enumerate}
\item[(i)] $
\int\limits_{R_0}^{+\infty}\dfrac{dx_1}{g^3(x_1)}<+\infty
$

{\rm or}
\item[(ii)] $
\int\limits_{R_0}^{+\infty}\dfrac{dx_1}{g^3(x_1)}=+\infty \ \ {\rm and} \ \ D=D^{(out)},
$
\end{enumerate}
then the problem \eqref{prad0} admits a weak solution ${\bf u}={\bf A}+{\bf v}$ in the sense of the definition\,\ref{d}. In the case $(i)$ the weak solution ${\bf u}$ satisfies the estimate
\begin{equation}\label{te1}\begin{array}{l}
\int\limits_{\Omega}|\nabla\textbf{u}|^2\,dx
\leq c({\bf a}, \|{\bf f}\|_*)\, \bigg(1+\int\limits_{R_0}^{+\infty}\dfrac{dx_1}{g^3(x_1)}\bigg)\end{array}\end{equation}
and in the case $(ii)$
\begin{equation}\label{te}\begin{array}{l}
\int\limits_{\Omega_{k}}|\nabla\textbf{u}|^2\,dx
\leq c({\bf a}, \|{\bf f}\|_*)\, \bigg(1+\int\limits_{R_0}^{R_k}\dfrac{dx_1}{g^3(x_1)}\bigg),\end{array}\end{equation}
where $c({\bf a}, \|{\bf f}\|_*)=c\bigg(\|{\bf a}\|^2_{W^{1/2,\,2}(\partial\Omega)}+\| {\bf a}\|^4_{W^{1/2,\,2}(\partial\Omega)}\!+\!\|{\bf f}\|^2_*\bigg)$ and $c$ is independent of $k.$
\end{theo}

\begin{remark}
In the equality \eqref{ws3} and in what follows we kept for simplicity the notation of $<{\bf f}, \bfeta>$ as an integral. One should also notice that due to the symmetry assumptions on ${\bf A},$ ${\bf v},$ ${\bf f}$ the equality \eqref{ws3} will hold as soon as it holds for any $\bfeta\in J^\infty_{0,S}(\Omega),$ i.e. for $\bfeta\in J^\infty_{0}(\Omega)$ which is symmetric.
\end{remark}
In order to prove the existence of at least one weak solution we need some classical results.

\begin{lem}\label{ls} (Leray-Schauder theorem). Let $V$ be a Hilbert space and $\mathcal{A}: V\rightarrow V$ be a nonlinear compact  operator. If the norms of all possible solutions of the operator equation $$u^{(\lambda)}=\lambda\mathcal{A}u^{(\lambda)}, \ \ \lambda\in [0,1],$$
are bounded with the same constant $c$ independent of $\lambda,$ i.e.,
$$\|u^{(\lambda)}\|_{V}\leq c \ \ \forall \lambda\in [0,1],$$
then the operator equation
$$u=\mathcal{A}u$$
has at least one solution $u\in V$ (see, for example, \cite{Lad}).
\end{lem}

\begin{lem}\label{poincare} {\sl (Poincar\'{e} inequality).} 
Let $u \in W^{1,2}_{loc}(\overline{\Omega}),$ $u\big|_{\partial\Omega}=0.$ Then the following inequality
\begin{equation}\label{poincareg}\begin{array}{l}
\int\limits_{\omega_k}|u(x)|^2\,dx\leq c\,g^2 (R_k)\int\limits_{\omega_k}|\nabla u(x)|^2\,dx, 
\end{array}
\end{equation}
holds, where the constant $c$ is independent of $u$ and $k.$
\end{lem}
For the proof of this lemma recall \eqref{g5}:
\begin{equation*}\begin{array}{l}
\dfrac{1}{2}g(R_k)\leq g(t)\leq \dfrac{3}{2}g(R_k), \ \ t\in [R_k, \ R_{k+1}].
\end{array}\end{equation*}

\begin{lem}\label{ladg}
Let $u \in W^{1,2}_{loc}(\overline{\Omega}),$ $u\big|_{\partial\Omega}=0.$ Then the following inequality
\begin{equation}\label{lad1g}\begin{array}{l}
\|u\|_{L^4(\omega_k)}\leq c\,g^{1/2}(R_k)\|\nabla u\|_{L^2(\omega_k)},
\end{array}
\end{equation}
holds, where the constant $c$ is independent of $u$ and $k.$
\end{lem}
\begin{proof}The proof of this lemma follows directly from the following inequality
\begin{equation}\label{mul}
\|u\|_{L^4(\omega_k)}\leq c \|\nabla u\|^{1/2}_{L^2(\omega_k)}\cdot\| u\|^{1/2}_{L^2(\omega_k)},
\end{equation}
the estimates \eqref{g5} and the Poincar\'{e} inequality \eqref{poincareg}. The constant $c$  in \eqref{mul} is independent of $k$.
\end{proof}

\begin{lem}\label{dulgg} Suppose that $D=D^{(out)},$ i.e.
$\omega_k=\{x: R_k<x_1<R_{k+1}, \ |x_2|<g(x_1) \}.$
Let $f\in L^2(\omega_k)$ and 
$$\begin{array}{rcl}
\int\limits_{\omega_k}f\,dx=0.
\end{array}$$
Then the problem
\begin{equation}\label{du}\left\{\begin{array}{rcl}
\div\, \textbf{u} & = & f \ \ {\rm in} \ \omega_k,\\
\textbf{u} & = & 0 \ \ {\rm on} \ \partial\omega_k
\end{array}\right\}\end{equation}
admits a solution $\textbf{u}\in W^{1,2}_0 (\omega_k)$ satisfying the estimate
\begin{equation}\label{dueh}
\|\nabla \textbf{u}\|_{L^2(\omega_k)}\leq c\,\| f \|_{L^2(\omega_k)}\end{equation}
with the constant $c$ independent of $\textbf{u}, \ f$ and $k.$
\end{lem}
\begin{remark}
{\rm In \cite{PilSol} the family of the domains $\omega_k$ was chosen in a special way in order to have solutions of the problem \eqref{du} satisfying the estimates \eqref{dueh} with a constant $c$ independent of $k.$ Below we give a detailed proof of that fact.}
\end{remark}
\begin{proof}
Recall that $R_{k+1}-R_k=\dfrac{g(R_k)}{2L}$ and $L$ is the Lipschitz constant of $g.$ Consider the transformation $F$ defined by
$$
y=(y_1, y_2)=F(x)=\bigg(\dfrac{2L\,(x_1-R_k)}{g(R_k)}, \, \dfrac{2L\,x_2}{g(R_k)}\bigg).
$$
Through this transformation $\omega_k$ is transformed into a domain $F(\omega_k)$ such that
$$
0\leq y_1=\dfrac{2L\,(x_1-R_k)}{g(R_k)}\leq \dfrac{2L\,(R_{k+1}-R_k)}{g(R_k)}=1,$$
$$
|y_2|\leq\dfrac{2L\,g(x_1)}{g(R_k)}=\dfrac{2L\,(g(x_1)-g(R_k)+g(R_k))}{g(R_k)}\leq 2L\,\bigg(\dfrac{L\,(R_{k+1}-R_k)}{g(R_k)}+1\bigg)=3L.$$
Moreover, the upper and the lower boundary of $F(\omega_k)$ is given by $\pm$ the graph of the function $h_k$ defined as
$$
h_k(y_1)=\dfrac{2L}{g(R_k)}\,g\Big(\dfrac{g(R_k)}{2L}\,y_1+R_k\Big), \ \ y_1\in (0,1).
$$
Note that $h_k$ satisfies
$$
|h_k(y_1)-h_k(y'_1)|\leq L|y_1-y'_1| \ \ \ \forall y_1, y'_1\in (0,1).
$$
Since $h_k(0)=2L$ it is clear that the graph of $h_k$ (resp. $-h_k$) is contained in the triangle $A^+B^+C^+$ (resp.  $A^-B^-C^-$)  (see Fig.\,3).
\begin{figure}[ht!]
\centering
\includegraphics[scale=0.9]{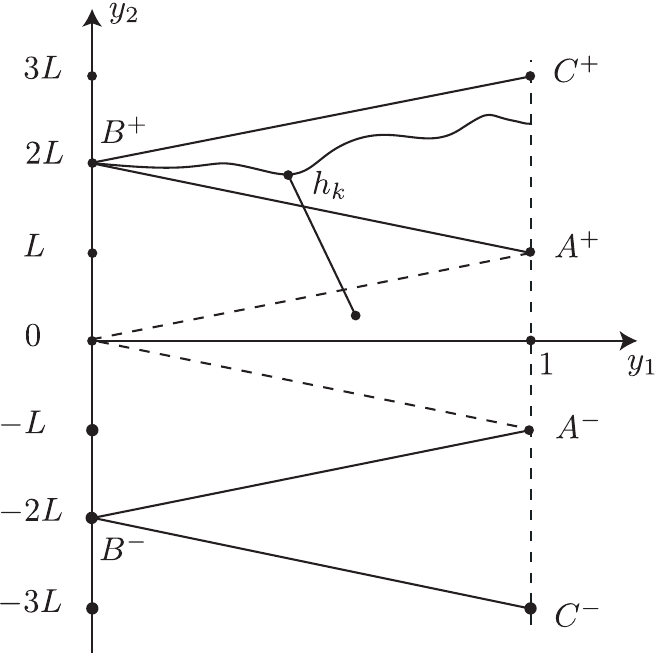}
\caption{Transformed domain}
\end{figure}
Any straight line joining a point of the triangle $A^-OA^+$ (notice that $O=(0,0)$) to the graph of $\pm h_k$ will necessarily have a slope larger than $L$ and thus $F(\omega_k)$ is a star shaped domain with respect to any point of $A^-OA^+$ and bounded independently of $k.$  One has if $J_F(x)$ denotes the Jacobian determinant of $F$ and $F^{-1}$ the inverse of $F$
\begin{equation*}\begin{array}{l}
\int\limits_{F(\omega_k)}f(F^{-1}(y))\,dy=\int\limits_{\omega_k}f(x)\,|J_F(x)|\,dx=\bigg(\dfrac{2L}{g(R_k)}\bigg)^2\int\limits_{\omega_k}f(x)\,dx=0.
\end{array}\end{equation*}
Thus there exists ${\bf v}$ solution to
\begin{equation*}\left\{\begin{array}{rcl}
\div\, \textbf{v}(y) & = & \dfrac{g(R_k)}{2L}\,f(F^{-1}(y)) \ \ {\rm in} \ F(\omega_k),\\
\textbf{v}(y) & = & 0 \ \ {\rm on} \ \partial F(\omega_k)
\end{array}\right\}\end{equation*}
which satisfies (see \cite{LadSol1})
\begin{equation}\label{divv}
\|\nabla \textbf{v}\|_{L^2(F(\omega_k))}\leq c\,\| \dfrac{g(R_k)}{2L}\,f(F^{-1}(y)) \|_{L^2(F(\omega_k))},\end{equation}
where $c$ is independent of $k.$ Set
$${\bf u}(x)={\bf v}(F(x)).$$
One has the summation convention
\begin{equation*}\begin{array}{l}
\dfrac{\partial u_k(x)}{\partial x_i}=\sum\limits_{l=1}^2\dfrac{\partial v_k(F(x))}{\partial y_l}\cdot\dfrac{\partial y_l}{\partial x_i}=\dfrac{2L}{g(R_k)}\cdot\dfrac{\partial v_k(F(x))}{\partial y_i}.
\end{array}\end{equation*}
Thus ${\bf u}$ satisfies 
\begin{equation*}\left\{\begin{array}{rcl}
\div\,\textbf{u}(x) & = & \dfrac{2L}{g(R_k)}\,\div\,{\bf v}(F(x))=f(x) \ \ {\rm in} \ \omega_k,\\
\textbf{u}(x) & = & 0 \ \ {\rm on} \ \partial \omega_k.
\end{array}\right\}\end{equation*}
Moreover,
\begin{equation*}
\|\nabla \textbf{u}\|_{L^2(\omega_k)}=\dfrac{2L}{g(R_k)}\,\|\nabla {\bf v}(F(x)) \|_{L^2(\omega_k)}.
\end{equation*}
Since (see \eqref{divv})
\begin{equation*}\begin{array}{l}
\|\nabla {\bf v}(F(x)) \|^2_{L^2(\omega_k)}=\int\limits_{\omega_k}|\nabla {\bf v}(F(x))|^2dx=\int\limits_{F(\omega_k)}|\nabla {\bf v}(y) |^2 \Big(\dfrac{g(R_k)}{2L}\Big)^2dy\\
\\
\leq c\,\Big(\dfrac{g(R_k)}{2L}\Big)^4\int\limits_{F(\omega_k)}f^2(F^{-1}(y))dy=c\,\Big(\dfrac{g(R_k)}{2L}\Big)^2\int\limits_{\omega_k}f^2(x)\,dx
\end{array}\end{equation*}
the result follows.
\end{proof}
\vspace{10pt}
\textit{Proof of the Theorem\,\ref{th}:}
we construct a solution to \eqref{ws3} as limit of a sequence ${\bf v}^{(l)}\in H_S(\Omega_l),$ where ${\bf v}^{(l)}$ are solutions to 
\begin{equation}\label{iv}\begin{array}{l}
\nu\int\limits_{\Omega_l} \nabla {\bf v}^{(l)} : \nabla\bfeta \, dx -
\int\limits_{\Omega_l} (({\bf A}+{\bf v}^{(l)})\cdot\nabla)\bfeta\cdot{\bf v}^{(l)} \, dx
-\int\limits_{\Omega_l} ({\bf v}^{(l)}\cdot\nabla)\bfeta\cdot{\bf A} \,
dx\\
\\
=\int\limits_{\Omega_l} ({\bf A}\cdot\nabla)\bfeta\cdot{\bf A} \, dx -
\nu\int\limits_{\Omega_l} \nabla {\bf A} : \nabla\bfeta \, dx+\int\limits_{\Omega_l}{\bf f}\cdot \bfeta\,dx
\end{array}\end{equation}
for any test function $\bfeta\in H_S(\Omega_l).$
Due, for instance, to the Riesz representation theorem there exits a unique element ${\mathcal{A}}\,{\bf v}^{(l)}\in H_S(\Omega_l)$ such that  
\begin{equation*}\begin{array}{l}
\int\limits_{\Omega_l}\nabla\widehat{\mathcal{A}}\,{\bf v}^{(l)}:\nabla\bfeta\,dx  = \nu^{-1}\bigg(\int\limits_{\Omega_l}({\bf v}^{(l)}\cdot\nabla)\,\bfeta\cdot {\bf v}^{(l)}\,dx+\int\limits_{\Omega_l}({\bf A}\cdot\nabla)\,\bfeta\cdot {\bf v}^{(l)}\,dx\\
\\
+\int\limits_{\Omega_l}({\bf v}^{(l)}\cdot\nabla)\,\bfeta\cdot {\bf A}\,dx
+\int\limits_{\Omega_l}({\bf A}\cdot\nabla)\,\bfeta\cdot {\bf A}\,dx+\int\limits_{\Omega_l}{\bf f}\cdot\bfeta\,dx \bigg)-\int\limits_{\Omega_l}\nabla {\bf A}:\nabla\bfeta\,dx \ \ \forall\bfeta\in H_S(\Omega_l).
\end{array}\end{equation*}
The equation \eqref{iv} is equivalent to the operator equation 
\begin{equation}\label{oe}\begin{array}{l}
{\bf v}^{(l)}=\widehat{\mathcal{A}}\,{\bf v}^{(l)}.
\end{array}\end{equation}
It can be proved (see \cite{Lad}) that the operator
$\widehat{\mathcal{A}}:H_S(\Omega_l)\hookrightarrow H_S(\Omega_l)$ 
is compact and the solvability of the operator equations \eqref{oe} can be obtained by applying the Leray--Schauder Theorem. To do this we need to show that the norms of all possible solutions of the operator equations 
\begin{equation}\label{oe1}\begin{array}{l}
{\bf v}^{(l,\,\lambda)}=\lambda\widehat{\mathcal{A}}\,{\bf v}^{(l,\,\lambda)}, \ \ \ \lambda\in[0,\,1],
\end{array}\end{equation}
are bounded by a constant independent of $\lambda.$ Take $\bfeta={\bf v}^{(l,\,\lambda)}$ in \eqref{oe1}. This yields
\begin{equation}\label{iden}\begin{array}{rcl}
\nu\int\limits_{\Omega_l}|\nabla{\bf v}^{(l,\,\lambda)}|^2\,dx &=&\lambda\int\limits_{\Omega_l}({\bf A}\cdot\nabla)\,{\bf v}^{(l,\,\lambda)}\cdot {\bf A}\,dx
-\lambda\,\nu\int\limits_{\Omega_l}\nabla{\bf A}:\nabla {\bf v}^{(l,\,\lambda)}\,dx\\
\\
&+&\lambda\int\limits_{\Omega_l}{\bf f}\cdot {\bf v}^{(l,\,\lambda)}\,dx
+\lambda\int\limits_{\Omega_l}({\bf v}^{(l,\,\lambda)}\cdot\nabla)\,{\bf v}^{(l,\,\lambda)}\cdot{\bf A}\,dx.
\end{array}
\end{equation}
\\We estimate the first three terms of the right-hand side of \eqref{iden} by using the H{\"o}lder and the Cauchy inequalities, and to estimate the last term of \eqref{iden} we use the Leray--Hopf inequality \eqref{A}. We obtain
\begin{equation}\label{est1}\begin{array}{l}
\nu\int\limits_{\Omega_l}|\nabla{\bf v}^{(l,\,\lambda)}|^2dx \leq c\,\mu\int\limits_{\Omega_l}|\nabla{\bf v}^{(l,\,\lambda)}|^2\,dx\\
\\
+\dfrac{c}{\mu}\bigg(\int\limits_{\Omega_l}|\nabla{\bf A}|^2dx +\int\limits_{\Omega_l}|{\bf A}|^4dx+\|{\bf f}\|^2_{H^{*}(\Omega_l)}\bigg)+c(\mathbb{F}_1,...,\mathbb{F}_N)\,\varepsilon\int\limits_{\Omega_l}|\nabla{\bf v}^{(l,\,\lambda)}|^2\,dx.
\end{array}\end{equation}
Taking $\mu=\dfrac{\nu}{4\,c}$ and $\varepsilon=\dfrac{\nu}{4\,c(\mathbb{F}_1,...,\mathbb{F}_N)},$ we obtain
\begin{equation*}\begin{array}{l}
\dfrac{\nu}{2}\|\nabla {\bf v}^{(l,\,\lambda)}\|^2_{L^2(\Omega_l)}\leq c\bigg(\|\nabla {\bf A}\|^2_{L^2(\Omega_l)} + \| {\bf A}\|^4_{L^4(\Omega_l)} + \|{\bf f}\|^2_{H^{*}(\Omega_l)}\bigg).
\end{array}
\end{equation*}
Since $\varepsilon$ is now fixed, we have also (note that ${\rm supp}\,{\bf B}_0\subset\Omega_0$)
\begin{equation}\label{123}\begin{array}{l}
\|\nabla {\bf A} \|^2_{L^2(\Omega_l)}=\|\nabla {\bf B}_0+\nabla {\bf B}_{\infty}\|^2_{L^2(\Omega_l)}=\|\nabla {\bf B}_0\|^2_{L^2(\Omega_0)}+\|\nabla {\bf B}_{\infty}\|^2_{L^2(\Omega_l)}\\
\\
\leq c\,\|{\bf a}\|^2_{W^{1/2,2}(\partial\Omega)}\bigg(\int\limits_{\Omega_0}dx
+\int\limits_{\Omega_l}\dfrac{dx}{g^4(x_1)}\bigg)\\
\\
\leq c\,\|{\bf a}\|^2_{W^{1/2,2}(\partial\Omega)}\bigg({\rm meas}(\Omega_0)
+\int\limits_{R_0}^{R_l}\int\limits_{-g(x_1)}^{g(x_1)}\dfrac{dx_1\,dx_2}{g^4(x_1)}\bigg)\\
\\
\leq c\,\|{\bf a}\|^2_{W^{1/2,2}(\partial\Omega)}\,\bigg(1+\int\limits_{R_0}^{R_l}\dfrac{dx_1}{g^3(x_1)}\bigg).
\end{array}
\end{equation}
Similarly (see \eqref{inf1})

\begin{equation}\label{1234}\begin{array}{l}
\|{\bf A} \|^4_{L^4(\Omega_l)}\leq c\,\|{\bf a}\|^4_{W^{1/2,2}(\partial\Omega)}\,\bigg(1+\int\limits_{R_0}^{R_l}\dfrac{dx_1}{g^3(x_1)}\bigg).
\end{array}
\end{equation}
The constant $c$ in \eqref{123} and \eqref{1234} is independent of $l.$
\\Therefore, we obtain for all $0\leq\lambda\leq 1$
\begin{equation*}\begin{array}{l}
\|\nabla {\bf v}^{(l,\,\lambda)}\|^2_{L^2(\Omega_l)}\leq c({\bf a}, \|{\bf f}\|_*)\,\bigg(1+\int\limits_{R_0}^{R_l}\dfrac{dx_1}{g^3(x_1)}\bigg).
\end{array}
\end{equation*}
Hence, according to the Leray--Schauder Theorem each operator equation \eqref{oe} has at least one weak symmetric solution ${\bf v}^{(l)}\in H_S(\Omega_l).$ These solutions satisfy the integral identity \eqref{iv} and the inequality
\begin{equation}\label{XX3}\begin{array}{l}
\|\nabla {\bf v}^{(l)}\|^2_{L^2(\Omega_l)}\leq c({\bf a}, \|{\bf f}\|_*)\,\bigg(1\!+\!\int\limits_{R_0}^{R_l}\dfrac{dx_1}{g^3(x_1)}\bigg).
\end{array}
\end{equation}

If $\begin{array}{l} \int\limits_{R_0}^{+\infty}\dfrac{dx_1}{g^3(x_1)}<+\infty,\end{array}$the right hand side of the above inequality is bounded by a constant uniformly independent of $l.$ Extending the solutions ${\bf v}^{(l)}$ by $0$ into $\Omega\setminus\Omega_l$ we get functions $\tilde{{\bf v}}^{(l)}\in H_S(\Omega).$
The sequence $\{\tilde{{\bf v}}^{(l)} \}$ is bounded in the space $H_S(\Omega).$ Therefore, there exists a subsequence $\{\tilde{{\bf v}}^{(l_m)}\}$ which converges weakly in $H_S(\Omega)$ and strongly\footnote{Notice that the embedding $H_S(\Omega_l)\!\hookrightarrow\! L^4_S(\Omega_l)$ is compact.} in $L^4(\Omega_l)$ for any $l.$ Taking in integral identity \eqref{iv}
an arbitrary test function $\bfeta$ with a compact support, we can find a number $l$ such that $\supp\bfeta\subset\Omega_l$ and $\bfeta\in
H_S(\Omega_l).$ We can  pass in \eqref{iv} to a limit as $l_m\to
+\infty$. As a result we get for the limit vector function ${\bf v}$ the
integral identity \eqref{ws3}. Obviously, following estimate
\begin{equation*}\label{te1}\begin{array}{l}
\int\limits_{\Omega}|\nabla\textbf{v}|^2\,dx
\leq c({\bf a}, \|{\bf f}\|_*)\,\bigg(1+\int\limits_{R_0}^{+\infty}\dfrac{dx_1}{g^3(x_1)}\bigg)\end{array}\end{equation*}
holds.

However, if  $\begin{array}{l} \int\limits_{R_0}^{+\infty}\dfrac{dx_1}{g^3(x_1)}=+\infty\end{array},$ we cannot pass to a limit because the right hand side of \eqref{XX3} is growing. Therefore, we have to control the Dirichlet integral of the vector field ${\bf v}^{(l)}$ over subdomains $\Omega_k\subset\Omega_l$, for $ k\leq l.$ To do this we apply the special techniques (so called estimates of Saint Venant type) developed by V.A. Solonnikov and O.A. Ladyzhenskaya (see \cite{LadSol3}, \cite{Sol3}). Let us estimate the norm $\|\nabla{\bf v}^{(l)}\|_{L^2(\Omega_k)}$ with $k<l.$
We introduce the function
\begin{equation}\label{p225}
\begin{array}{l}
\textbf{U}_k^{(l)}(x)=\begin{cases}
\textbf{v}^{(l)}(x), \ \ x\in\Omega_k,\\
\theta_k (x)\textbf{v}^{(l)}(x)+\widehat{\textbf{v}}^{(l)}_k (x), \ \ x\in\omega_{k} ,\\
0, \ \ x\in\Omega\setminus\Omega_{k+1},
\end{cases}
\end{array}
\end{equation}
where $\theta_k(x)$ is a smooth even in $x_2$ cut-off function with the following properties:
$$\theta_k (x)=\begin{cases}
1, \ \ x\in\Omega_{k},\\
0, \ \ x\in\Omega\setminus\Omega_{k+1},
\end{cases}$$
\begin{equation}\label{teta}|\nabla\theta_k(x)|\leq \dfrac{c}{g (R_k)}.\end{equation}
Let $\widehat{\textbf{v}}^{(l)}_k$ be a solution of the problem
\begin{equation}\label{p2525}
\begin{array}{rcl}
\div\,\widehat{\textbf{v}}^{(l)}_k & = & -\nabla\theta_k\cdot\textbf{v}^{(l)} \qquad \hbox{\rm in }\;\;\omega_{k},\\
\widehat{\textbf{v}}^{(l)}_k &  = & 0  \qquad \hbox{\rm on }\;\;\partial\omega_{k}.
 \end{array}
 \end{equation} 
Since 
$$\begin{array}{l}
\int\limits_{\omega_{k}}\nabla\theta_k\cdot\textbf{v}^{(l)}\,dx=\int\limits_{\omega_{k}}\div\,(\theta_k \textbf{v}^{(l)})\,dx=\int\limits_{\partial\omega_{k}}\theta_k \textbf{v}^{(l)}\cdot\textbf{n}\,dx=\int\limits_{\sigma(R_k)}\textbf{v}^{(l)}\cdot\textbf{n}\,dx=0,\end{array}$$
a solution $\widehat{\textbf{v}}^{(l)}_k$ of problem \eqref{p2525} exists and satisfies the estimate

\begin{equation}\label{p25}
\begin{array}{rcl}
\|\nabla\widehat{\textbf{v}}^{(l)}_k\|_{L^2(\omega_{k})} &  \leq & c \|\nabla\theta_k\cdot\textbf{v}^{(l)}\|_{L^2(\omega_{k})},
 \end{array}
 \end{equation}
where $c$ is independent of $k$ (see Lemma\,\ref{dulgg}).
Using the estimate \eqref{teta} and the Poincar\'{e} inequality \eqref{poincareg}, from \eqref{p25} we derive the estimate
\begin{equation}\label{p26}
\begin{array}{l}
\|\nabla\widehat{\textbf{v}}^{(l)}_k\|_{L^2(\omega_{k})}\!\leq \!c \|\nabla\theta_k\cdot\textbf{v}^{(l)}\|_{L^2(\omega_{k})}\!\leq\!\dfrac{c}{g (R_k)} \|\textbf{v}^{(l)}\|_{L^2(\omega_{k})}\!\leq\!c\|\nabla\textbf{v}^{(l)}\|_{L^2(\omega_{k})}.
\end{array}
\end{equation} 
Notice that $\widehat{\textbf{v}}^{(l)}_k$ is not necessary symmetric, so we symmetrized it as in \eqref{sym}. For simplicity we do not change the notation of $\widehat{\textbf{v}}^{(l)}_k,$ i.e. $\widehat{\textbf{v}}^{(l)}_k$ is symmetric in the following text.

Set $\bfeta=\textbf{U}_k^{(l)}$ in \eqref{iv}. Then, because ${\bf U}_k^{(l)}=0$ in $\Omega\setminus\Omega_{k+1}$ and
\begin{equation*}\begin{array}{l}
\int\limits_{\Omega_{k+1}}((\textbf{v}^{(l)}+\textbf{A})
\cdot\nabla)\textbf{U}_k^{(l)}\cdot\textbf{U}_k^{(l)}\,dx=0,
\end{array}\end{equation*}
we obtain
\begin{equation}\label{p2526}\begin{array}{l}
\nu\int\limits_{\Omega_{k}}|\nabla \textbf{v}^{(l)}|^2\,dx=\int\limits_{\omega_k}((\textbf{v}^{(l)}+\textbf{A})\cdot\nabla)\,\textbf{U}_k^{(l)}\cdot(\textbf{v}^{(l)}
-\textbf{U}_k^{(l)})\,dx\\
\\
-\nu\int\limits_{\omega_k}\nabla \textbf{v}^{(l)}:\nabla\textbf{U}_k^{(l)}\,dx
+\int\limits_{\Omega_{k+1}} (\textbf{v}^{(l)}\cdot\nabla)\,\textbf{U}_k^{(l)}\cdot\textbf{A}\,dx\\
\\
-\nu\int\limits_{\Omega_{k+1}}\nabla\textbf{A}:\nabla\textbf{U}_k^{(l)}\,dx
+\int\limits_{\Omega_{k+1}} (\textbf{A}\cdot\nabla)\,\textbf{U}_k^{(l)}\cdot\textbf{A}\,dx+\int\limits_{\Omega_{k+1}}{\bf f}\cdot\textbf{U}^{(l)}_k\,dx.
\end{array}\end{equation}
\\In order to estimate the right hand side of \eqref{p2526}, we use first the inequalities \eqref{p26}, \eqref{lad1g} and the Poincar\'{e} inequality \eqref{poincareg} to obtain
\begin{equation}\label{extra}\begin{array}{rcl}\|\textbf{v}^{(l)}\|_{L^4(\omega_k)} &\leq & c\, g^{1/2}(R_k) \|\nabla\textbf{v}^{(l)}\|_{L^2(\omega_k)};\\
\\
\|\textbf{v}^{(l)}\!\!-\!\!\textbf{U}^{(l)}_k\|_{L^4(\omega_k)}&\leq &\|\textbf{v}^{(l)}\|_{L^4(\omega_k)}+\|\widehat{\textbf{v}}^{(l)}_k\|_{L^4(\omega_k)}\\
\\
&\leq & c\, g^{1/2}(R_k) \|\nabla\textbf{v}^{(l)}\|_{L^2(\omega_k)}+ c\,g^{1/2}(R_k)\|\nabla\widehat{\textbf{v}}_k^{(l)}\|_{L^2(\omega_k)}\\
\\
&\leq & c\, g^{1/2}(R_k) \|\nabla\textbf{v}^{(l)}\|_{L^2(\omega_k)};\\
\\
\|\nabla\textbf{U}^{(l)}_k\|_{L^2(\omega_k)}&\leq &\|\nabla(\theta_k\textbf{v}^{(l)})\|_{L^2(\omega_k)}+\|\nabla\widehat{\textbf{v}}^{(l)}_k\|_{L^2(\omega_k)}\\
\\
&\leq &\|\nabla\theta_k\|_{L^{\infty}(\omega_k)}\,\|\textbf{v}^{(l)}\|_{L^2(\omega_k)}\!+\!\|\theta_k\|_{L^{\infty}(\omega_k)}\,\|\nabla\textbf{v}^{(l)}\|_{L^2(\omega_k)}\!\!\\
\\
& + & c\, \|\nabla\textbf{v}^{(l)}\|_{L^2(\omega_k)}
\leq c\, \|\nabla\textbf{v}^{(l)}\|_{L^2(\omega_k)}.\end{array}\end{equation}
Below we will need the following inequality
\begin{equation}\label{LF}\begin{array}{l}
\int\limits_{\omega_k}|{\bf A}|^2|{\bf w}|^2dx\leq c\,\varepsilon^2\int\limits_{\omega_k}|\nabla {\bf w}|^2dx \ \ \ \ \forall {\bf w}\in W^{1,2}_{loc}(\overline{\Omega}), \ {\bf w}=0 \ {\rm on} \ \partial\Omega,
\end{array}\end{equation}
which can be proved arguing like for proving Leray-Hopf's inequality.
\\By using the H{\"o}lder inequality, \eqref{extra} and \eqref{LF} we obtain
\begin{equation}\label{1est}\begin{array}{l}
\bigg|\int\limits_{\omega_k}((\textbf{v}^{(l)}+\textbf{A})
\cdot\nabla)\textbf{U}_k^{(l)}\cdot(\textbf{v}^{(l)}-\textbf{U}_k^{(l)})\,dx \bigg|\\
\\
\leq \Big(\|\textbf{v}^{(l)}\|_{L^4(\omega_k)}
\|\textbf{v}^{(l)}-\textbf{U}^{(l)}_k\|_{L^4(\omega_k)}
+ \|\textbf{A}\,(\textbf{v}^{(l)}-\textbf{U}^{(l)}_k)\|_{L^2(\omega_k)}\Big) \|\nabla\textbf{U}^{(l)}_k\|_{L^2(\omega_k)}
\\
\\
\leq c\,g(R_k)\,\|\nabla {\bf v}^{(l)} \|^3_{L^2(\omega_k)}
+c\,\varepsilon\,\|\nabla {\bf v}^{(l)} \|_{L^2(\omega_k)}\,\|\nabla( {\bf v}^{(l)}-{\bf U}^{(l)}_k )\|_{L^2(\omega_k)}\\
\\
\leq c\,g(R_k)\,\|\nabla {\bf v}^{(l)} \|^3_{L^2(\omega_k)}+c\,\varepsilon\,\|\nabla {\bf v}^{(l)} \|^2_{L^2(\omega_k)}.
\end{array}\end{equation}
We estimate the second term of the equation \eqref{p2526} by using the Cauchy-Schwarz inequality and the estimates \eqref{extra}:
\begin{equation}\label{2est}\begin{array}{l}
\nu\bigg|\int\limits_{\omega_k}\nabla \textbf{v}^{(l)}:\nabla\textbf{U}_k^{(l)}\,dx \bigg|\leq \nu\|\nabla\textbf{v}^{(l)}\|_{L^2(\omega_k)}\, \|\nabla\textbf{U}^{(l)}_k\|_{L^2(\omega_k)}\leq \nu c \|\nabla\textbf{v}^{(l)}\|^2_{L^2(\omega_k)}.
\end{array}\end{equation}
To estimate the third term of \eqref{p2526} we use  the Leray-Hopf inequality \eqref{A}, the H{\"o}lder inequality, \eqref{extra} and \eqref{LF}:
\begin{equation}\label{3est}\begin{array}{l}
\bigg| \int\limits_{\Omega_{k+1}} (\textbf{v}^{(l)}\cdot\nabla)\textbf{U}^{(l)}_k \cdot \textbf{A}\, dx  \bigg|
\\
\\
\leq \bigg| \int\limits_{\Omega_{k}} (\textbf{v}^{(l)}\cdot\nabla)\textbf{v}^{(l)} \cdot \textbf{A}\, dx  \bigg|+\bigg| \int\limits_{\omega_{k}} (\textbf{v}^{(l)}\cdot\nabla)\textbf{U}^{(l)}_k \cdot \textbf{A}\, dx  \bigg|
\\
\\
\leq c\,\varepsilon\,\|\nabla\textbf{v}^{(l)}\|^2_{L^2(\Omega_{k})}+ \|\nabla\textbf{U}^{(l)}_k\|_{L^2(\omega_{k})}\, \bigg(\int\limits_{\omega_{k}}|\textbf{v}^{(l)}|^2 |\textbf{A}|^2\,dx\bigg)^{1/2}\\
\\
\leq c\,\varepsilon\,\bigg( \|\nabla\textbf{v}^{(l)}\|^2_{L^2(\Omega_{k})}+
\|\nabla\textbf{v}^{(l)}\|^2_{L^2(\omega_{k})} \bigg).
\end{array}\end{equation}
The last three terms of \eqref{p2526} are estimated by using the H{\"o}lder inequality, the Cauchy inequality, \eqref{123}, \eqref{1234} and \eqref{extra}:

\begin{equation}\label{4est}\begin{array}{l}
\nu\bigg|\int\limits_{\Omega_{k+1}}\nabla\textbf{A}:\nabla\textbf{U}^{(l)}_k\,dx\bigg|+\bigg|\int\limits_{\Omega_{k+1}} (\textbf{A}\cdot\nabla)\textbf{U}^{(l)}_k\cdot\textbf{A}\,dx\bigg|+
\bigg|\int\limits_{\Omega_{k+1}}{\bf f}\cdot {\textbf{U}^{(l)}_k\,dx  \bigg|}
\\
\\
\leq c\bigg( \|\nabla\textbf{A}\|_{L^2(\Omega_{k+1})}+ \|\textbf{A}\|^2_{L^4(\Omega_{k+1})} +\|{\bf f}\|_{H^{*}(\Omega_{k+1})}\bigg)\|\nabla\textbf{U}^{(l)}_k\|_{L^2(\Omega_{k+1})}\\
\\
\leq\dfrac{c}{2\,\mu}\bigg(\|\nabla\textbf{A}\|_{L^2(\Omega_{k+1})}\!+\!\|\textbf{A}\|^2_{L^4(\Omega_{k+1})}\!+\!\|{\bf f}\|_{H^{*}(\Omega_{k+1})}\bigg)^2\!+\!\dfrac{c\,\mu}{2}\|\nabla\textbf{U}^{(l)}_k\|^2_{L^2(\Omega_{k+1})}\\
\\
\!\leq \dfrac{2\,c}{\mu}\,\bigg(\|\nabla\textbf{A}\|^2_{L^2(\Omega_{k+1})}+ \|\textbf{A}\|^4_{L^4(\Omega_{k+1})}+\|{\bf f}\|^2_{H^{*}(\Omega_{k+1})}\bigg)+\dfrac{c\,\mu}{2}\|\nabla\textbf{v}^{(l)}\|^2_{L^2(\Omega_{k+1})}\\
\\
\leq \dfrac{c({\bf a}, \|{\bf f}\|_*)}{\mu}\,\bigg( 1+\int\limits_{R_0}^{R_{k+1}}\dfrac{dx_1}{g^3 (x_1)} \bigg)+c\,\dfrac{\mu}{2}\,\bigg(\|\nabla\textbf{v}^{(l)}\|^2_{L^2(\Omega_{k})}+
\|\nabla\textbf{v}^{(l)}\|^2_{L^2(\omega_{k})}\bigg).
\end{array}\end{equation}
\\Therefore, from \eqref{p2526}, \eqref{1est}, \eqref{3est}, \eqref{4est} it follows that
\begin{equation*}\begin{array}{l}
\nu\,\int\limits_{\Omega_{k}}|\nabla \textbf{v}^{(l)}|^2\,dx\leq c \,g(R_k)\|\nabla\textbf{v}^{(l)}\|^3_{L^2(\omega_k)}+c\,\varepsilon \,\|\nabla\textbf{v}^{(l)}\|^2_{L^2(\omega_k)}
+ c\,\nu\,\|\nabla\textbf{v}^{(l)}\|^2_{L^2(\omega_k)}\\
\\
+c\,(\varepsilon+\dfrac{\mu}{2})\,\bigg( \|\nabla\textbf{v}^{(l)}\|^2_{L^2(\Omega_{k})}+
\|\nabla\textbf{v}^{(l)}\|^2_{L^2(\omega_{k})} \bigg)
+\dfrac{c({\bf a}, \|{\bf f}\|_*)}{\mu}\,\bigg( 1+\int\limits_{R_0}^{R_{k+1}}\dfrac{dx_1}{g^3 (x_1)} \bigg).
\end{array}\end{equation*}
For $\varepsilon$ and $\mu$ sufficiently small, we obtain
\begin{equation}\label{lastt}\begin{array}{l}
\int\limits_{\Omega_{k}}|\nabla \textbf{v}^{(l)}|^2\,dx \leq  c \, g(R_k)\|\nabla\textbf{v}^{(l)}\|^3_{L^2(\omega_k)}+ c \|\nabla\textbf{v}^{(l)}\|^2_{L^2(\omega_k)}\\
\\+ c({\bf a}, \|{\bf f}\|_*)\,\bigg( 1+\int\limits_{R_0}^{R_{k+1}}\dfrac{dx_1}{g^3 (x_1)} \bigg).
\end{array}\end{equation}
Using the remark\,\ref{remark1} several times we derive
\begin{equation*}\begin{array}{l}
\int\limits_{R_k}^{R_{k+1}}\dfrac{dx_1}{g^3(x_1)}\leq 
\int\limits_{R_k}^{R_{k+1}}\dfrac{dx_1}{\big(\frac{1}{2}\,g(R_k)\big)^3}=\dfrac{8\,(R_{k+1}-R_k)}{g^3(R_k)}=\dfrac{4}{L\,g^2(R_k)},
\end{array}\end{equation*}
\begin{equation*}\begin{array}{l}
\int\limits_{R_{k-1}}^{R_k}\dfrac{dx_1}{g^3(x_1)}\geq 
\int\limits_{R_{k-1}}^{R_k}\dfrac{dx_1}{\big(\frac{3}{2}\,g(R_{k-1})\big)^3}=\dfrac{8\,(R_{k}-R_{k-1})}{27\,g^3(R_{k-1})}=\dfrac{4}{27\,L\,g^2(R_{k-1})}.
\end{array}\end{equation*}
Since $g(R_k)\geq \dfrac{1}{2}g(R_{k-1})$ we get
\begin{equation*}\begin{array}{l}
\int\limits_{R_{k-1}}^{R_k}\dfrac{dx_1}{g^3(x_1)}\geq 
\dfrac{1}{27\,L\,g^2(R_k)}.
\end{array}\end{equation*}
It follows that
\begin{equation*}\begin{array}{l}
\int\limits_{R_k}^{R_{k+1}}\dfrac{dx_1}{g^3(x_1)}\leq \dfrac{4}{L\,g^2(R_k)}=27\cdot 4\dfrac{1}{27\,L\,g^2(R_k)}\leq 
27\cdot 4\,\int\limits_{R_{k-1}}^{R_{k}}\dfrac{dx_1}{g^3(x_1)}.
\end{array}\end{equation*}
Thus we have
\begin{equation*}\begin{array}{l}
\int\limits_{R_0}^{R_{k+1}}\dfrac{dx_1}{g^3(x_1)}\leq 
\int\limits_{R_0}^{R_{k}}\dfrac{dx_1}{g^3(x_1)}+
\int\limits_{R_k}^{R_{k+1}}\dfrac{dx_1}{g^3(x_1)}\leq 109\,\int\limits_{R_0}^{R_{k}}\dfrac{dx_1}{g^3(x_1)}
\end{array}\end{equation*}
and the inequality \eqref{lastt} becomes
\begin{equation*}\begin{array}{l}
\int\limits_{\Omega_{k}}|\nabla \textbf{v}^{(l)}|^2\,dx \leq  c \, g(R_k)\|\nabla\textbf{v}^{(l)}\|^3_{L^2(\omega_k)}+ c \|\nabla\textbf{v}^{(l)}\|^2_{L^2(\omega_k)}+ c({\bf a}, \|{\bf f}\|_*)\,\bigg( 1+\int\limits_{R_0}^{R_{k}}\dfrac{dx_1}{g^3 (x_1)} \bigg).
\end{array}\end{equation*}
\\Denote $y_k=\int\limits_{\Omega_{k}}|\nabla \textbf{v}^{(l)}|^2\,dx.$ Since $\int\limits_{\omega_k}=\int\limits_{\Omega_{k+1}}-\int\limits_{\Omega_{k}}$, we can rewrite the last inequality as
\begin{equation}\label{XX4}\begin{array}{l}
y_k\leq c_*(y_{k+1}-y_k)+c_{**}g(R_k)(y_{k+1}-y_k)^{3/2}+\dfrac{1} {2}Q_k,
\end{array}\end{equation}
where
\begin{equation}\label{q}\begin{array}{l}
Q_k=2\,c({\bf a}, \|{\bf f}\|_*)\, \bigg( 1+\int\limits_{R_0}^{R_k}\dfrac{dx_1}{g^3 (x_1)} \bigg).
\end{array}\end{equation}
We have, using  the remark\,\ref{remark1} again

\begin{equation*}\begin{array}{l}
c_*(Q_{k+1}-Q_k)+c_{**}g(R_k)(Q_{k+1}-Q_k)^{3/2}\\
\\
\leq 2\,c_*\,c({\bf a}, \|{\bf f}\|_*)\int\limits_{R_k}^{R_{k+1}}\dfrac{dx_1}{g^3 (x_1)}
+c_{**}\,g(R_k)\bigg( c({\bf a}, \|{\bf f}\|_*)\int\limits_{R_k}^{R_{k+1}}\dfrac{dx_1}{g^3 (x_1)} \bigg)^{3/2}
\end{array}\end{equation*}
\begin{equation*}\begin{array}{l}
\leq c_1\,c({\bf a}, \|{\bf f}\|_*)\,g^{-2}(R_k)
\leq c_2\,c({\bf a}, \|{\bf f}\|_*) \,\int\limits_{R_{k-1}}^{R_{k}}\dfrac{dx_1}{g^3 (x_1)}

\leq c({\bf a}, \|{\bf f}\|_*) \,\Big(1+\int\limits_{R_{0}}^{R_{k}}\dfrac{dx_1}{g^3 (x_1)}\Big)=\dfrac{1}{2}Q_k
\end{array}\end{equation*}
for $k$ large enough if 
$\dfrac{\int\limits_{R_{k-1}}^{R_{k}}\dfrac{dx_1}{g^3 (x_1)}}{1+\int\limits_{R_{0}}^{R_{k}}\dfrac{dx_1}{g^3 (x_1)}}\rightarrow 0$
when $k\rightarrow +\infty.$

{\bf Claim.} {\it Let non negative numbers $y_k, \ k=1,...,N,$ satisfy the inequalities $y_{k+1}\geq y_k$ and
\begin{equation}\label{XX1}\begin{array}{l}
y_k\leq c_*(y_{k+1}-y_k)+c_{**}\, g(R_k) (y_{k+1}-y_k)^{3/2}+\dfrac{1}{2}Q_k,
\end{array}\end{equation}
where $c_*, c_{**}, Q_k$ are non negative numbers such that
\begin{equation}\label{XX2}\begin{array}{l}
\dfrac{1}{2}Q_k\geq c_*(Q_{k+1}-Q_k)+c_{**}\, g(R_k) (Q_{k+1}-Q_k)^{3/2}.
\end{array}\end{equation}
If $N<+\infty$ and $y_N\leq Q_N$
then $y_k\leq Q_k \ \ \forall k<N$.}
\\

Although this claim is proved in \cite{Sol3}, for the reader convenience we give the proof which is based on induction.  Suppose we have proved that $y_{k+1}\leq Q_{k+1}.$ If $y_k>Q_k$ then $0\leq y_{k+1}-y_k<Q_{k+1}-Q_k.$
Since the function $\tau\rightarrow F(\tau)=c_*\tau+c_{**}\,g(R_k)\,\tau^{3/2}$ is increasing we get
$$\begin{array}{l} 
y_k\leq F(y_{k+1}-y_k)+\dfrac{1}{2}Q_k<F(Q_{k+1}-Q_k)+\dfrac{1}{2}Q_k\leq\dfrac{1}{2}Q_k+\dfrac{1}{2}Q_k=Q_k
\end{array}$$
and a contradiction. Thus, $y_k\leq Q_k.$

Since $Q_k$ satisfies the condition \eqref{XX2}, the inequality \eqref{XX4} together with \eqref{XX3} and the claim above, the estimate
\begin{equation}\label{final}\begin{array}{l}
y_k = \int\limits_{\Omega_k}|\nabla{\bf v}^{(l)}|^2\,dx
\leq c({\bf a}, \|{\bf f}\|_*)\,\bigg(1+\int\limits_{R_0}^{R_{k}}\dfrac{dx_1}{g^3 (x_1)}\bigg) \ \ \forall k\leq l
\end{array}\end{equation}
holds.

Since for every bounded domain $\Omega_k$,  $k>0$ the embedding $W_S^{1,2}(\Omega_k)\!\hookrightarrow\! L^4_S(\Omega_k)$ is compact, the estimate \eqref{final} is sufficient to assure the existence of a subsequence $\{\textbf{v}^{(l_m)}\}$ which converges weakly 
in $\mathring{W}^{1,2}_S (\Omega_k)$ and strongly in  $L^4_S(\Omega_k)$ for any $k>0.$  Such subsequence could be constructed by Cantor diagonal process: we can choose a weakly convergent subsequence $\{\textbf{v}^{(l_m)}\}$ in 
$\mathring{W}^{1,2}_S (\Omega_1)$ which converges strongly in $L^4_S(\Omega_1).$ In the same manner we subtract a subsequence of $\{\textbf{v}^{(l_m)}\}$ in $\Omega_2$ which we call also $\{\textbf{v}^{(l_m)}\}$ for the sake of simplicity. Continuing this we can choose a desired subsequence. Taking in integral identity \eqref{iv}
an arbitrary test function $\bfeta$ with a compact support, we can find a number $k$ such that $\supp\bfeta\subset\Omega_k$ and $\bfeta\in
H_S(\Omega_k).$ Extending $\bfeta$ by zero into $\Omega\setminus
\Omega_k$, and considering all integrals in \eqref{iv} as integrals
over $\Omega$, we can  pass in \eqref{iv} to a limit as $l_m\to
+\infty$. As a result we get for the limit vector function ${\bf v}$ the
integral identity \eqref{ws3}. 
Therefore, $\textbf{u}=\textbf{A}+\textbf{v}$ is a weak solution of problem \eqref{prad0}. The estimate \eqref{te} for $\textbf{v}$ follows from \eqref{final}. Since for $\textbf{A}$ the analogous to \eqref{te} is also valid, we obtain \eqref{te} for the sum $\textbf{u}=\textbf{A}+\textbf{v}$.

\begin{remark}
{\rm If the norms $\|{\bf a}\|_{W^{1/2,2}(\partial\Omega)}$ and $\|{\bf f} \|_*$ are sufficiently small, it can be proved using the methods proposed in \cite{LadSol3} and \cite{Sol3} that the weak solution ${\bf u}$ is unique in a class of functions with the Dirichlet integral growing ``not too fast".}
\end{remark}

\begin{remark}
{\rm If $D$ is a channel-like outlet and $|\mathbb{F}|$ is sufficiently small, it can be proved using the methods from \cite{LadSol3} and \cite{Sol3} that the weak solution ${\bf u}$ tends to the Poiseuille flow as $x_1\rightarrow +\infty.$ In this sense our result extends the result obtained by H. Morimoto and H. Fujita in \cite{Morimoto1}, \cite{Morimoto2}.}
\end{remark}

\section*{Acknowledgement}

The research of M. Chipot, K. Pileckas and W. Xue leading to these results has received funding from Lithuanian-Swiss cooperation programme to reduce economic and social disparities within the enlarged European Union under project agreement No. CH-3-SMM-01/01.
\\The research of M. Chipot and W. Xue was supported by the Swiss National Science Foundation under the grant \#200021-146620.
\\The research of K. Kaulakyt\.{e} was completed during a visit at the University of Zurich supported by Sciex - Scientific Exchange Programme NMS-CH.


\begin{thebibliography}{1}
{\scriptsize

\bibitem{Amick}{\sc Ch.J. Amick}: Existence of Solutions to the
Nonhomogeneous Steady Navier--Stokes Equations, {\sl Indiana Univ.
Math. J.\/} {\bf 33} (1984), 817--830.

\bibitem{BOPI}{\sc W. Borchers and K. Pileckas}: Note on the Flux
Problem for Stationary Navier--Stokes Equations in Domains with
Multiply Connected Boundary, {\sl Acta App. Math.\/} {\bf 37}
(1994), 21--30.

\bibitem{Finn}{\sc R. Finn}: On the Steady-State Solutions of the
Navier--Stokes Equations. III, {\sl Acta Math.} {\bf 105} (1961),
197--244.

\bibitem{Fu}{\sc H. Fujita}: On the Existence and Regularity of the
Steady-State Solutions of the Navier--Stokes Theorem, {\sl J. Fac.
Sci. Univ. Tokyo Sect.\/} I (1961) {\bf 9}, 59--102.

\bibitem{Fu1} {\sc H. Fujita}: On Stationary Solutions to Navier--Stokes Equation in
Symmetric Plane Domain under General Outflow Condition, {\sl
Pitman Research Notes in Mathematics, Proceedings of International
Conference on Navier--Stokes Equations. Theory and Numerical
Methods. June 1997. Varenna, Italy\/} (1997) {\bf 388}, 16-30.

\bibitem{FM}{\sc  H. Fujita and H. Morimoto}: A Remark on the Existence
of the Navier--Stokes Flow with Non-Vanishing  Outflow Condition,
{\sl GAKUTO Internat. Ser. Math. Sci. Appl.\/} {\bf 10} (1997),
53--61.

\bibitem{Galdi}{\sc  G.P. Galdi}: {\sl An Introduction to the Mathematical Theory of the Navier--Stokes Equations: Steady-State Problems (second edition),} Springer (2011).

\bibitem{Galdi1}{\sc  G.P. Galdi}: On the Existence of Steady
Motions of a Viscous Flow with Non-Homogeneous Conditions, {\sl
Le Matematiche\/} {\bf 66} (1991), 503--524.

\bibitem{KAU}{\sc  K. Kaulakyt\.{e}}: On Nonhomogeneous Boundary Value Problem for the Steady Navier-Stokes System in Domain with Paraboloidal and Layer Type Outlets to Infinity, {\sl Topological Methods in Nonlinear Analysis}, accepted (2015).

\bibitem{KP}{\sc  K. Kaulakyt\.{e} and K. Pileckas}: On the Nonhomogeneous Boundary Value Problem for the Navier--Stokes System in a Class of Unbounded Domains, {\sl J. Math. Fluid Mech.}, {\bf 14,} No. 4 (2012), 693-716.

\bibitem{KPR} {\sc M.V. Korobkov, K. Pileckas and R. Russo}: On the
Flux Problem in the Theory of Steady Navier--Stokes Equations with
Nonhomogeneous Boundary Conditions, {\sl
Arch. Rational Mech. Anal.},  {\bf 207,} No. 1 (2013), 185-213.

\bibitem{KPR1} {\sc M.V. Korobkov, K. Pileckas and R. Russo}: Steady Navier--Stokes System with Nonhomogeneous Boundary Conditions
in the Axially Symmetric Case, {\sl C. R. Mecanique} {\bf 340} (2012), 115--119.

\bibitem{KPR0} {\sc M.V. Korobkov, K. Pileckas and R. Russo}: Solution of Leray's Problem for Stationary Navier--Stokes Equations in Plane and Axially
Symmetric Spatial Domains, {\sl Annals of Mathematics} {\bf 181} (2015), 769--807.

\bibitem{KPR5} {\sc M.V. Korobkov, K. Pileckas and R. Russo}: The Existence Theorem for Steady Navier-Stokes Equations in the Axially Symmetric Case, {\sl Annali della Scuola Normale Superiore di Pisa Classe di Scienze} {\bf 14,} No.1 (2015), 233--262.

\bibitem{Kozono}{\sc H. Kozono and T. Yanagisawa}: Leray's Problem on the
Stationary Navier--Stokes Equations with Inhomogeneous Boundary Data,
{\sl Math. Z. \/} {\bf 262} No. 1 (2009), 27--39.

\bibitem{Lad} {\sc O.A. Ladyzhenskaya}: {\sl The Mathematical Theory
of Viscous  Incompressible Flow,} Gordon and Breach (1969).

\bibitem{Lad1} {\sc O.A. Ladyzhenskaya}: Investigation of the
Navier--Stokes Equation for Stationary Motion of an
Incompressible Fluid, {\sl Uspech Mat. Nauk\/}  {\bf 14}, No. 3 (1959),
75--97 (in Russian).

\bibitem{LadSol1} {\sc O.A. Ladyzhenskaya and V.A. Solonnikov}:
Some Problems of Vector Analysis and Generalized Formulations
of Boundary Value Problems for the Navier--Stokes Equations, {\sl
Zapiski Nauchn. Sem. LOMI \/} {\bf 59} (1976), 81--116. English
Transl.: {\sl J. Sov. Math.\/} {\bf 10}, No. 2 (1978), 257--285.

\bibitem{LadSol3}
{\sc O.A. Ladyzhenskaya and V.A. Solonnikov}: Determination of the
solutions of boundary value problems for stationary Stokes and
Navier-Stokes equations having an unbounded Dirichlet integral,
{\sl Zapiski Nauchn. Sem. LOMI\/} {\bf 96} (1980), 117--160.
English Transl.: {\sl J. Sov. Math.\/}, {\bf 21}, No. 5 (1983),
728--761.

\bibitem{Leray}
{\sc J. Leray}:  \'Etude de diverses
\'equations int\'egrales non lin\'eaire  et de quelques
probl\`emes que pose l'hydrodynamique, {\sl J. Math. Pures
Appl.\/} {\bf 12} (1933), 1--82

\bibitem{Morimoto1}
{\sc H. Morimoto and H. Fujita}: A Remark on the Existence of Steady Navier--Stokes
Flows in 2D Semi-Infinite Channel Infolving the General Outflow Condition,
{\sl Mathematica Bohemica\/} {\bf 126}, No. 2 (2001), 457--468.

\bibitem{Morimoto2}
{\sc H. Morimoto and H. Fujita}: A Remark on the Existence of Steady Navier--Stokes
Flows in a Certain Two-dimensional Infinite Channel,
{\sl Tokyo Journal of Mathematics\/} {\bf 25}, No. 2 (2002), 307--321.

\bibitem{Morimoto3}
{\sc H. Morimoto and H. Fujita}: Stationary Navier--Stokes Flow in 2-Dimensional
Y-Shape Channel Under General Outflow Condition,
{\sl The Navier--Stokes Equations: Theorey and Numerical Methods, Lecture Note in
Pure and Applied Mathematics, Marcel Decker (Morimoto Hiroko ,Other)\/} {\bf 223}, (2002), 65--72.

\bibitem{Morimoto4}
{\sc H. Morimoto}: Stationary Navier--Stokes Flow in 2-D Channels Involving the General Outflow Condition,
{\sl Handbook of Differential Equations: Stationary Partial Differential Equations\/}
{\bf 4}, Ch. 5, Elsevier (2007), 299--353.

\bibitem{Morimoto}
{\sc H. Morimoto}: A Remark on the Existence of 2-D Steady Navier--Stokes
Flow in Bounded Symmetric Domain Under General Outflow Condition,
{\sl J. Math. Fluid Mech.\/} {\bf 9}, No. 3 (2007), 411--418.

\bibitem{NazPil3}
{\sc S.A. Nazarov and K. Pileckas}: On the Solvability of the Stokes and
Navier--Stokes Problems in Domains that are Layer-Like at
Infinity, {\sl J. Math. Fluid Mech.\/} {\bf 1}, No. 1 (1999),
78-116.

\bibitem{Neustupa1}
{\sc J. Neustupa}: On the Steady Navier--Stokes Boundary Value Problem
in an Unbounded $2D$ Domain with Arbitrary Fluxes Through the Components of the Boundary,
{\sl Ann. Univ. Ferrara\/}, {\bf 55}, No. 2 (2009), 353--365.

\bibitem{Neustupa}
{\sc J. Neustupa}: A New Approach to the Existence of Weak Solutions of the Steady
Navier--Stokes System with Inhomoheneous Boundary Data in Domains with Noncompact Boundaries,
{\sl Arch. Rational Mech. Anal\/} {\bf 198}, No. 1 (2010), 331--348.


\bibitem{Pukhnachev} {\sc  V.V. Pukhnachev}: Viscous Flows in Domains with a Multiply Connected
Boundary, {\sl New Directions in Mathematical Fluid Mechanics. The
Alexander V. Kazhikhov Memorial Volume. Eds. Fursikov A.V., Galdi
G.P. and Pukhnachev V.V., Basel -- Boston -- Berlin: Birkhauser\/}
(2009) 333--348.

\bibitem{Pukhnachev1} {\sc  V.V. Pukhnachev}: The Leray Problem
and the Yudovich Hypothesis, {\sl Izv. vuzov. Sev.--Kavk. region.
Natural sciences. The special issue "Actual Problems of
Mathematical Hydrodynamics"\/} (2009), 185--194 (in Russian).


\bibitem{Sazonov}
{\sc  L.I. Sazonov}, On the Existence of a Stationary Symmetric
Solution of the Two-Dimensional Fluid Flow Problem, {\sl Mat.
Zametki\/} {\bf 54}, No. 6 (1993), 138--141. English Transl.: {\sl
Math. Notes}, {\bf 54}, No. 6 (1993), 1280--1283.

\bibitem{PilSol}
{\sc V.A. Solonnikov and K. Pileckas}: Certain Spaces of
Solenoidal Vectors and the Solvability of the Boundary Value
Problem for the Navier--Stokes  System of Equations in Domains with
Noncompact Boundaries, {\sl Zapiski Nauchn. Sem. 
LOMI\/} {\bf 73} (1977), 136--151. English Transl.: {\sl J. Sov.
Math.} {\bf 34}, No. 6 (1986), 2101--2111.

\bibitem{Sol3}
{\sc V.A. Solonnikov}: Stokes and Navier--Stokes Equations in
Domains with Noncompact Boundaries, {\sl Nonlinear Partial
Differential Equations and Their Applications. Pitmann Notes in
Math., College de France Seminar\/} {\bf 3} (1983), 240-349.

\bibitem{Stein}{\sc E.M. Stein}:{\sl Singular Integrals and Differentiability Properties of Functions,}
Prinston University Press (1970)

\bibitem{Tak}
{\sc A. Takeshita}: A Remark on Leray's
Inequality,  {\sl Pacific J. Math.\/} {\bf 157} (1993), 151--158.

\bibitem{VorJud} {\sc I.I. Vorovich and V.I. Judovich}: Stationary
Flows of a Viscous Incompressible Fluid, {\sl Mat. Sbornik\/} {\bf
53} (1961), 393--428 (in Russian).
}
\end{thebibliography}
 \end{document}